\documentclass[11pt]{article}
\usepackage[margin=2.5cm]{geometry}
\usepackage{color,ulem}
\usepackage{amsmath,amssymb,amsthm}
\usepackage{fancyhdr}
\usepackage{psfrag}
\usepackage{array}
\usepackage{graphicx}
\usepackage[utf8]{inputenc}
\usepackage{tikz}
\usetikzlibrary{positioning,arrows}
\tikzset{
block/.style={
  draw, 
  rectangle, 
  minimum height=1.0cm, 
  minimum width=2cm, align=center
  }, 
line/.style={->,>=latex'}
}
\usepackage{hyperref}
\usepackage{bbold}
\usepackage{cases}
\usepackage{enumitem}

\definecolor{orange}{cmyk}{0,0.5,1,0.3}
 \newcommand{\comM}[1]{\textcolor{black}{#1}}

\newtheorem{lemma}{Lemma}
\newtheorem{theorem}{Theorem}
\newtheorem{proposition}{Proposition}
\newtheorem*{acknowledgments}{Acknowledgments}
\newtheorem{definition}{Definition}

\newtheorem{remark}{Remark}

\def\dst{\displaystyle}
\def\eps{\varepsilon}

\def\p{\partial}
\def\a{\alpha}

\def\g{\gamma}

\def\({\Bigl (}
\def\){\Bigr )}
\newcommand{\be}{\begin{equation}}
\newcommand{\1}{{\mathchoice {\rm 1\mskip-4mu l} {\rm 1\mskip-4mu l}{\rm 1\mskip-4.5mu l} {\rm 1\mskip-5mu l}}}
\newcommand{\f}{\frac}
\newcommand{\Log}{\text{Log}}
\newcommand{\Arg}{\text{Arg}}
\newcommand{\ee}{\end{equation}}
\newcommand{\bea}{$$ \begin{array}{lll}}
\newcommand{\eea}{\end{array} $$}
\newcommand{\bi}{\begin{itemize}}
\newcommand{\ei}{\end{itemize}}

\numberwithin{equation}{section}

\newtheorem{prop}{Proposition}

\DeclareMathOperator{\R}{{\mathbb R}}
\DeclareMathOperator{\C}{{\mathbb C}}
\DeclareMathOperator{\Z}{{\mathbb Z}}

\providecommand{\eps}{\varepsilon}

\renewenvironment{proof}{\noindent{\bf Proof.}}{\hfill
  $\blacksquare$\par\noindent}

\date{}

\graphicspath{{figs/}}

\title{{Estimating the division rate and kernel in the fragmentation equation}}
\begin{document}


\author{Marie Doumic
 \thanks{
INRIA Rocquencourt,  équipe-projet MAMBA, domaine de Voluceau, BP 105, 78153 Rocquencourt,
France. Email: marie.doumic@inria.fr.} 
\and Miguel Escobedo 
\thanks{
Universidad del Pa\'is Vasco, Facultad de Ciencias y Tecnolog\'ia, Departamento de Matem\'aticas, Barrio
Sarriena s/n 48940 Lejona (Vizcaya), Spain. Email: miguel.escobedo@ehu.es.}
\and Magali Tournus 
\thanks{Centrale Marseille, I2M, UMR 7373, CNRS, Aix-Marseille univ., Marseille, 13453, France. Email: magali.tournus@centrale-marseille.fr.}
}
\maketitle
\begin{abstract}
{We consider the fragmentation equation 
$$\dfrac{\p f}{\p t}(t,x) =-  B(x)f(t,x) +  \dst\int_{y=x}^{y=\infty} k(y,x)B(y)f(t,y) dy,$$
and address the question of estimating the fragmentation parameters -{ \it i.e.} the division
rate $B(x)$ and the fragmentation kernel  $k(y,x)$ - from measurements  of the size distribution $f(t,\cdot)$ at various times. 
This is a natural question for any application where the sizes of the particles are measured experimentally whereas the fragmentation rates are unknown, see for instance (Xue, Radford, Biophys. Journal, 2013) for amyloid fibril breakage.}
{Under the assumption of a polynomial division rate $B(x)=\alpha x^\gamma$ and a self-similar fragmentation kernel $k(y,x)=\f{1}{y}k_0(\f{x}{y})$, we use the asymptotic behaviour proved in (Escobedo, Mischler, Rodriguez-Ricard, Ann. IHP, 2004) to obtain uniqueness of the triplet $(\alpha,\gamma,k_0)$ and a representation formula for $k_0$. To }{invert} {this formula, 
}
{one of the  }{delicate points is to prove that the Mellin transform of the asymptotic profile never vanishes, what we do through  the \textcolor{black}{use of the Cauchy integral}. 
}
\end{abstract}




\section{Introduction}

This paper presents a theoretical study about the identification of the {functional} parameters 
of  the continuous fragmentation equation. {There are many possible applications to this problem, {\color{black} e.g.  flocculation~\cite{bortz2015inverse}, mining industry~\cite{bertoin2005}, bacterial growth~\cite{robert:hal-00981312}. T}his question {\color{black}having} first emerged from the application to amyloid fibril breakage~\cite{XR13}, let us first explain briefly the motivation which guided this study}.


In { the biophysical article} \cite{XR13}, the authors study how a set of { aggregates of proteins (called} amyloid fibrils) behaves when undergoing turbulent agitation.
In vivo, amyloid fibrils break apart because of both enzymatic processes and agitation.
The question addressed in \cite{XR13} is to determine what is the effect of agitation on the 
fragmentation {\it rate} (what is the probability that a fibril of given length breaks apart?) 
and fragmentation kernel ({\it where} a fibril is more likely to break apart?) of proteins of 
amyloid types. The method used is the identification of the parameters of a model describing 
the fragmentation process {by minimising a least squares functional, which represented the discrepancy between the experimental measurements and the model outputs}{, after parametrization of the problem. The best-fit model which came as an outcome in~\cite{XR13} happened to fall into the scope of assumptions where the asymptotic behaviour of the fragmentation equation has been thorougly studied (see~\cite{EMR05} and the details provided in Section~\ref{sec:asymp}): a power law for the fragmentation rate, and a self-similar form for the fragmentation kernel. Our question was then: can these asymptotic results be used to estimate, in a non-parametric way, the fragmentation kernel and the power law of the fragmentation rate?}

{ This leading idea - how to use measurements on the asymptotic distribution to estimate functional parameters of the equation - has been first} initiated in~\cite{PZ} and continued e.g. in~\cite{BDE14,DPZ09} for the growth-fragmentation equation. { However, up to now, the studies  were focused on the question of estimating the division~{\it rate}, whereas the division {\it kernel} was assumed to be known. As shown below, estimating the fragmentation kernel reveals much more difficult, being a {\it severely ill-posed} inverse problem.}

{{ To our knowledge, very few studies exist on the question of estimating the division kernel from measurements of fragmenting particles.} In~\cite{hoang}, {{V.H. Hoang} estimated the fragmentation rate on a growth-fragmentation process, but assumed much richer data since the sizes of daughter and mother particles are measured at each time of division - such precise measurements may be possible for growing and dividing individuals such as bacteria, but  not for particles or polymers. { In \cite{hoang:hal-01623403}, V. H. Hoang et al. investigate a problem much closer to ours, for a growth-fragmentation equation with linear growth and constant fragmentation rate, taking into account a statistical treatment for the noise, but taking for granted the validity of a reconstruction formula of the same type as ours (see below Theorem~\ref{Well_posedness_inverse_pb}, (iii)). Finally, we can also cite the least square approach used i}n~\cite{bortz2015inverse}, { which also studies} the question of estimating the fragmentation kernel  on a more general dynamical system. { The authors} keep the time dependency of the equation and use a least squares approach on the cumulative distribution function.}


\subsection{Assumptions, notations and asymptotic behaviour}
\label{sec:asymp}
Fragmentation processes describe the mechanisms by 
which particles can break apart into smaller pieces.
In the simplest fragmentation models,
the particles are fully identified by their size (or volume
or number of elementary particles), which is a positive real number for continuous
models.
The fragmentation equation for one dimensional particles (or linear particles) 
describes the evolution
of  the density  $f(t,x) \geq 0$ of particles of 
size $x \in \mathbb{R}^+$ at time $t \geq 0$.
In the continuous setting, the pure fragmentation equation is written
\begin{equation}\label{eq:frag}
\left\{
 \begin{aligned}
\dfrac{\p f}{\p t}(t,x)& =-  B(x)f(t,x) +  \dst\int_{x}^{\infty} k(y,x)B(y)f(t,y) dy,\\
  f(x,0) & = f_0(x).
 \end{aligned}
 \right.
 \end{equation}
 The expression {\it pure} fragmentation is to be understood here as a contrast with the growth-fragmentation equation \cite{SM15}
or the coagulation-fragmentation equation~\cite{LM2}.

 Equation \eqref{eq:frag} expresses that particles of size $x$ break apart 
 into smaller pieces with a fragmentation rate $B(x)$. { The kernel describes the probability distribution of the  mass of the pieces formed in each fragmentation event, assuming that such a fragmentation event takes place.}
 The kernel $k$ then satisfies 
 \begin{equation}
k(y,x) = 0 \;\;\;\text{for}\;\;\; x>y,  \qquad \dst\int_{0}^{y} x k(y,dx)  = 1.  
 \end{equation}
{ In the case of {\it binary} fragmentation, {\it i.e.} when the particle breaks into exactly two parts, we moreover have a symetry property: $k(y,dx)=k(y,y-dx)$, which leads to the fact that $\int\limits_0^y k(y,dx)=2.$}
The parameters $B$ and $k$ are commonly chosen according to the particular physical 
process being modelled. Here we do the following assumptions - which happen to be satisfied by the best-fit model in the biological study~\cite{XR13}.
\begin{enumerate}[label=Hyp-\arabic*]
\item \label{hyp1} The fragmentation rate $B(x)$ follows a power law: $B(x)=\alpha x^{\gamma},$ with $\gamma>0$ {and $\alpha>0$}.

Since $\g>0$, a long filament is more likely to break apart than a smaller one. 
Let us point out the specific case $\gamma=1$ for which the probability that a filament breaks apart is directly proportional to its size.

\item \label{hyp2}The kernel $k(y,x)$ has a self-similar form, {\it i.e.}  the abscissa $x$ where the protein filaments of size $y$
are likely to break up depends only on the ratio $x/y$. More specifically, 
we  assume that there exists a bounded, non negative measure $k_0$ such that
\begin{equation}
\label{S1E1}
k(y,x)=\frac{1}{y}k_0\left(\frac{x}{y}\right)
\end{equation}
and 
\begin{equation}
\label{S1E2}
supp \left(k_0\right) \subset [0, 1],\,\,\,  \int\limits_0^1dk_0(z) <+\infty, \quad  \int\limits_0^1 z dk_0(z)=1.
\end{equation}

\item There exists $\eps>0$ such that $k_0$ is a bounded continuous function on $[1-\eps,1]$ and on $[0,\eps]$, \label{hyp4}
\item \label{hyp3}
$\exists \eps>0, \quad 0< \eta_1 < \eta_2 <1 $ such that$ \quad k_0(z) \geq \eps , \quad z \in [\eta_1,\eta_2].$
\end{enumerate}
A frequent example is  to take $dk_0(z) = \kappa \1_{(0,1)}(z)dz$ with $\kappa>1$ ($\kappa=2$ for binary fragmentation): this { may be interpreted as the fact that}  for { linear filaments, { any location along a given}  filament  has an equal 
probability to break apart.}

{Let us comment briefly our assumptions. We need } Assumption~\eqref{hyp4} {to prove Theorem~\ref{Well_posedness_inverse_pb} below}. We could also have relaxed it, as in~\cite{balague:hal-00683148,DG10}, replacing it by the following
\begin{equation}\exists \nu >0, \quad C>0\; \text{s.t. }\; \int\limits_0^x k_0(dz) \leq C x^\nu.\label{hyp:DG}
\end{equation}
However, this would reveal of no practical interest when tackling real data, and in order to avoid useless technical developments we stick here to Assumption~\eqref{hyp4}.

Assumption~\eqref{hyp3} is the assumption~(2.12) in~\cite{MMP05}, under which we have the asymptotic self-similar behaviour that we recall below. {Although}  this assumption is not useful in itself for the results contained in our study, all our approach relies on this asymptotic behaviour.

\

For the sake of clarity, let us now  rewrite the equation \eqref{eq:frag} where  $B$ and $k$ are replaced by their specific expression.
\begin{equation}\label{eq:frag_rewritten}
\left\{
 \begin{aligned}
\dfrac{\p f}{\p t}(t,x)& =-  \a x^{\g} f(t,x) + \a \dst\int_{x}^{\infty}y^{\g-1} k_0\left(\frac xy \right)f(t,y) dy,\\
  f(0,x) & = f_0(x).
 \end{aligned}
 \right.
 \end{equation}
 
 Under  assumptions~\eqref{hyp1}, \eqref{hyp2} {and~\ref{hyp3}}, {if $f_0 \in L^1\big(\R^+,(1+x^{1+m})dx\big),$ with $m>1$,} Theorem 3.{2}. in~\cite{MMP05} states that the fragmentation equation~\eqref{eq:frag_rewritten} has a {unique}  solution
 in ${ \mathcal{C}([0,T);{L}^1(\mathbb{R}^+,xdx))
\cap L^1(0,T;L^1(\R^+,x^{\gamma+{m}} dx)}$, where we define
$$L^1(\R^+, \mu):=\biggl\{f:\R^+\to \R,\quad \int\limits_0^\infty \vert f(x)  \vert d\mu (x) <\infty\biggr\}, \qquad L^1(\R^+)=L^1(\R^+,dx).$$
It can be seen by formal integration that equation \eqref{eq:frag} preserves the mass of the system, and that the total number of protein filaments
 is increasing with time
\begin{equation}
\label{mass_conservation}
 \begin{aligned}
  \dfrac{d}{dt} \dst\int_0^{+ \infty} f(t,x) dx& = \dst\int_0^{+ \infty} B(x)f(t,x)dx, \qquad \text{number of clusters increases,}\\
    \dfrac{d}{dt} \dst\int_0^{+ \infty} xf(t,x) dx& =0,  \qquad\qquad \qquad  \qquad \qquad \text{mass conservation}.\\
 \end{aligned}
\end{equation}

Let us point out that there may exist non-preserving-mass solutions: uniqueness is true only in ${ \mathcal{C}([0,T);{L}^1(\mathbb{R}^+,xdx))
\cap L^1(0,T;L^1(\R^+,x^{\gamma+m} dx)}$.
In particular, the fact that the solution is in $L^1(0,T;L^1(\R^+,x B(x) dx)$ is crucial.
For instance, for $B(x) = \dfrac{x}{2} (1+x)^{-r}, \; r \in (0,1)$, $k_0 = 2 \mathbb{1}_{[0,1]}$, and for the initial condition $f_0  \in L^1(\mathbb{R^+}, xdx) \setminus L^1(\mathbb{R^+}, x B(x)dx)$ defined as
\begin{equation}
 f_0(x) =  \exp \left(- {\dst\int_0^x} \dfrac{3-ru (1+u)^{-1}}{2 \lambda (1+u)^r + u}du\right),
\end{equation}
the author of \cite{Stewart90} points out that there is a solution 
$f(t,x)= \exp(\lambda t) f_0(x)$ belonging to $\mathcal{C}([0,T);L^1(\mathbb{R}^+,(1+x)dx))$ and for which the total mass of the system increases exponentially fast.

Besides the well-posedness properties, the qualitative behaviour of the fragmentation equation was also deeply explored, {\it e.g.}~\cite{balague:hal-00683148,EMR05,MMP05,SM15}. Under assumptions \eqref{hyp1} \eqref{hyp2} \eqref{hyp3}, it has been proven (Theorem~3.2. in~\cite{MMP05}) that the solution $f(t,x)$ satisfies
\begin{equation}
\label{asbe}
\lim\limits_{t\to \infty} \int\limits_0^{+ \infty} x\big\vert f(t, x) -  t^{\f{2}{\gamma}}
g\big(x  t^{\f{1}{\gamma}}\big)\big\vert dx = 0,
\end{equation}
where {under the extra assumption~\ref{hyp4}} the self-similar profile $g$ satisfies {(Theorem 1 of~\cite{DG10})}:

\begin{eqnarray}
\label{S1E10}
\forall k\ge 0,\,\,\,x^{k+1} g\in W^{1,1}(0, \infty),\quad x^{k} g \in L^\infty\cap L^1(0, \infty),\quad \,\, g \in W^{1,\infty} _{ loc  }(0, \infty)
 \end{eqnarray}
 and is the unique solution in $L^1 _{ loc }(0, \infty)\cap L^1(xdx)$ to
\begin{eqnarray}\label{eq:g}
z g'(z) + (2 + \a \gamma z^{\gamma})  g(z)&=& \a \gamma \dst\int_{z}^{\infty} \dfrac{1}{u} k_0\left(\frac{z}{u}\right) u^{\gamma} g(u) \;du,\,\,\,\hbox{in}\,\,\,\mathcal D'(0, \infty)\\
 \int_0^\infty zg(z)dz &=&\rho.\label{eq:gBIS}
\end{eqnarray}

Notice that it follows from \eqref{S1E10} that:
\begin{eqnarray}
\label{S1E10B}
 g \in L^1(x^k dx),\,\,\forall k\ge 0.
 \end{eqnarray}
{Assumption~\ref{hyp4} is necessary to have $k\geq 0$ in~\eqref{S1E10B}, without assumption~\ref{hyp4} we only have $g\in L^1(x^kdx)$ for $k\geq 1,$ see~\cite{MMP05}}.

Since the total mass of the solutions that we consider is preserved, the parameter $\rho$ is determined by the 
initial condition: $\rho: = \dst\int_0^{+ \infty}x f_0(x) dx$.
Existence and uniqueness of \eqref{eq:g} are detailed in \cite{EMR05}, see also \cite{M1,DG10}. 

{Under additional assumptions, a polynomial rate of convergence for the limit~\eqref{asbe} (corresponding to an exponential rate of convergence for the growth-fragmentation equation, {\it i.e.} a spectral gap) is obtained in~\cite{CCM,balague:hal-00683148,SM15}.
Here we do not work under these supplementary assumptions.
} 
\subsection{Formulation of the inverse problem}
Let us assume that, as in~\cite{XR13}, the observation consists in measurements of the sizes of polymers/aggregates in samples taken at different times. { We discuss briefly in the conclusion how measurements - and their noise - may be further modeled, and in this paper we take for granted that from such measurements we are able to obtain estimates of the density distribution at several times.} 


The first inverse problem then states as follows: Given  measurements of $f(t_i,x)$ solution to the equation~\eqref{eq:frag_rewritten} at various times $t_i,$ is it possible to estimate the functional parameters of~\eqref{eq:frag_rewritten}, namely the triplet $(\alpha,\gamma,k_0) \in (0,+\infty)\times (0,+\infty) \times {\cal M}^+ (0,1)$?

Following the idea of~\cite{PZ}, we can take advantage of the asymptotic results recalled above, and, for {\it sufficiently large} times $t$, consider that measuring $f(t,x)$ provides a measurement of the asymptotic profile $g$ defined as the unique solution to~\eqref{eq:g} in $ L^1((1+x^{\gamma +1})dx).$ We are then led to reformulate the inverse problem as follows:

\

\begin{center}
\begin{minipage}{0.8\textwidth}
{\bf Inverse Problem (IP):} 
Given  a measurement of $g(x)$ solution in $ L^1((1+x^{\gamma +1})dx)$ of Equation~\eqref{eq:g}, is it possible to estimate the functional parameters of~\eqref{eq:g}, namely the triplet $(\alpha,\gamma,k_0) \in (0,+\infty)\times (0,+\infty) \times {\cal M}^+ (0,1)$?
\end{minipage}
\end{center}

\

\noindent
Though this formulation strongly recalls the one obtained in previous studies~\cite{PZ,BDE14} for the estimation of the division {\it rate} (here assumed to be given by the simple parametric form $\alpha x^\gamma$), estimating the division {\it kernel} $k_0$ reveals a much harder and more ill-posed problem. 


{Expliciting general conditions under  which, for a given function $g$ satisfying~\eqref{S1E10}, there exists a triplet $(\alpha,\;\gamma,\;k_0)\in \R^{+*}\times\R^{+*}\times {\cal M}^+(0,1)$ such that  $g$ satisfies~\eqref{eq:g} is an interesting and difficult question in itself, however not directly useful to solve ${\bf (IP)}$, for which we can {\it assume} this existence - as it is often the case in the field of inverse problems - rather than prove it.}

{ In this article, as a first theoretical step towards solving ${\bf (IP)}$, we thus focus on two important properties: uniqueness (Theorem~\ref{S2T1}), and how we can characterise such a triplet provided it exists (Theorem~\ref{Well_posedness_inverse_pb}). 
Stability with respect to a noisy measurement of $g$ has to be searched in a convenient regularity space, and, as the formulation $(iii)$ of Theorem~\ref{Well_posedness_inverse_pb} shows a severe ill-posedness of the problem, this stability is expected to be very weak. Together with numerical solution, this will be the subject of future research. } 

\section{Main results}

{Our method strongly relies on the Mellin transform of the equation, which, as in~\cite{BDE14}, appears to be a somewhat intrinsic feature of the equation: as shown below, it provides us with an explicit formulation of the kernel $k_0$ in terms of the inverse Mellin transform of a functional of the Mellin transform of $g$.}

\subsection{Formulation of the stationary equation \eqref{eq:g} in Mellin coordinates}

We  first recall the definition of the Mellin transform. 
{
\begin{definition}
 Let $\mu$ be a measure over $\R^+$. We denote by $M[\mu]$ the Mellin transform of $\mu$, defined by the integral
 \begin{equation}
 M[\mu] (s) = \dst\int_{0}^{+ \infty} x^{s-1}\; d\mu(x),
\end{equation}
for those values of $s$ for which the integral exists.
\end{definition}
\begin{remark}
 If the integral exists for some $u \in \R$, then, it converges for $s \in u + i \R$. If the integral exists for $u$ and $v$ in $\R$, $(u<v)$
 then it exists for $w \in (u,v)$.
Thus, in general, the Mellin transform of a measure is defined in a vertical band of the complex plane.
\end{remark}
}
Since it has been proven in \cite{DG10} that for all $k \geq 0$, $x^{k+1} g \in W^{1,1}(\mathbb{R}^+)$ and $x^k g \in L^\infty(\R^+),$
 we may define 
 \begin{equation}
G(s):=M[g](s), \qquad K_0(s):=M[k_0](s),\qquad \Re e(s) {\geq 1}.
\end{equation}

We now use the Mellin transform in  equation \eqref{eq:g} to obtain a non-local functional equation.
\begin{prop} 
\label{Prop2}Suppose that a function $g$,  satisfying properties  \eqref{S1E10}, 
is a solution of  the equation \eqref{eq:g} in the sense of distributions on $(0, \infty)$, for some kernel $k_0$ compactly supported in $[0, 1] $ and satisfying \eqref{S1E2}. Then, its Mellin transform:
\begin{equation}
\label{S2EG1}
G(s)=\int_0^\infty x^{s-1}g(x)dx
\end{equation}
is well defined {and analytic} for all $s\in \C$ such that $\Re e(s)\ge 1$, and it satisfies
 \begin{equation}\label{Formula}
(2-s) G(s) =  \a \g (K_0(s)-1) G(s+\g)\,\,\,\quad \forall s\in \C,\quad \Re e(s){>1}.
 \end{equation}
\end{prop}
\begin{remark}
\label{S2RFi}
As far as  $K_0(s)\not =1$, the  equation \eqref{Formula} may also be written as
 \begin{eqnarray}
&&G(s+\gamma )=\Phi (s)G(s) \label{Formula20},\qquad {\text{with }}\\
&&\Phi (s):=\frac {2-s} {\alpha \gamma (K_0(s)-1)}\label{Formula22}.
\end{eqnarray}
Since the support of the measure $k_0$ is contained in $[0, 1]$, it follows from  \eqref{S1E2} that $\vert K_0(s)\vert<1$ {for $\Re e(s)>2$}. We deduce that $\Phi (s)$ is analytic, and then the equation  \eqref{Formula} is equivalent to \eqref{Formula20}, in that same region {$\Re e(s)>2$}.
\end{remark}
\begin{proof}
The analyticity property of $G$ in  $D_1$ follows from the property $x^{r}g\in L^1(0, \infty)$ for all $r\ge 0$ thanks to~\eqref{S1E10}.  By the integrability properties of $g$, the Mellin transform may be applied to both sides of \eqref{eq:g}.
In particular,  for all $s>1$:
\begin{equation}
\label{S2E489}
\int _0^\infty xg'(x) x^{s-1}dx=-s\int _0^\infty g(x) x^{s-1}dx=-sG(s).
\end{equation}

On the other hand, using Fubini's Theorem, that may be applied due to the integrability properties  \eqref{S1E10} of $g$  and the hypothesis on $k_0$, we obtain:
\begin{eqnarray}
\int _0^\infty x^{s-1} \dst\int_{x}^{\infty} \dfrac{1}{u} k_0\left(\frac{x}{u}\right) u^{\gamma} g(u) \;dudx &=&
\int _0^\infty u^\gamma  g(u)\int _0^u x^{s-1} k_0\left(\frac{x}{u}\right)dxdu \nonumber\\
&=&\int _0^\infty u^{s+\gamma-1}  g(u)\int _0^1 y^{s-1} k_0\left(y\right)dydu\nonumber\\
&=& K_0(s)G(s+\gamma).\label{S2E4898}
\end{eqnarray}
It immediately  follows from \eqref{S2E489} and \eqref{S2E4898}  that $G$ satisfies \eqref{Formula}. 
\end{proof}
 
\subsection{Uniqueness of the fragmentation rate and kernel}

We use the formulation~{\eqref{Formula}} of Proposition~\ref{Prop2} to prove uniqueness of the parameters $\alpha$, $\gamma $ and of the measure $k_0$. {The assumptions on $g$ to obtain uniqueness are fairly general}, as stated in our first theorem.

\begin{theorem}[Uniqueness of a triplet solution to the inverse problem]
\label{S2T1}
For any  {nonnegative} function $g$ satisfying 
\begin{equation}
\label{S2Emom}
x^{k+1} g\in {W^{1,1}}(0, \infty), \quad x^k  g {\in L^1(0,\infty)}, \quad \forall k \ge 1,
\end{equation}
 there exists at most one triplet $(\gamma, \alpha, k_0)
\in \mathbb{R}^+ \times \mathbb{R}^+ \times \mathcal{M}(0,1)$ where  $k_0$ is a non negative measure satisfying \eqref{hyp2} and \eqref{hyp4}, such that $g$ is a solution of \eqref{eq:g} in the sense of distributions.
\label{thm_wp}
\end{theorem}

{
As said above, this is only a uniqueness result: the function g must satisfy many more conditions that are not completely general
to really be the long time asymptotics of the fragmentation equation, and then for the parameters $\gamma, \alpha,$ and the measure $k_0$ to exist. }

{Section~\ref{Uniqueness_gamma} is dedicated to the proof of {{Theorem \ref{thm_wp}}. Let us here briefly comment   the result.} The identity \eqref{asbe} {shows} that the profile $g$ and the parameter $\g$  describe the long time behaviour of the solution $f$ to \eqref{eq:frag_rewritten}.

According to Theorem~\ref{thm_wp}, if such $\g$ and $k_0$ exist they are  "encoded" in the equilibrium profile $g$.
\textcolor{black}{In other words, based { only} on the knowledge of the  asymptotic profile $g$, it is theoretically possible to {obtain}  information {on} the whole dynamics of the solution  $f$.}

 \textcolor{black}{The uniqueness of $\gamma$ is based on the characterisation given in Proposition~\ref{S3P3} of Section~\ref{Uniqueness_gamma}, which  uses the asymptotic behaviour of some  functional of $G(s)$ when $\Re e(s) \to \infty.$
 Equivalently, this is linked to the behaviour of $g$ for large $x,$ see also~\cite{balague:hal-00683148}. 
Since we cannot measure experimentally such a behaviour, th{e theoretical} characterisation {of Theorem \ref{thm_wp}}  
cannot be used in practice  to estimate its value.
To extract the values of $\g$ from real measurements, we would need  to use another characterization, based for instance on the knowledge of the {time evolution of the first moment of the size distribution}, given by~\eqref{asbe}, as described in Section~\ref{sec:asymp}. } 

Once the two parameters $\alpha $ and $\gamma $ are proved to be unique,  the uniqueness of the  measure $k_0$ is deduced from  general properties of the Mellin transform, and does not give any constructive method to calculate its values. The next subsection is thus dedicated to a constructive characterisation of $k_0$.

\subsection{Reconstruction of the fragmentation kernel}
Once $\alpha $ and $\gamma $ are known, we may apply Proposition~\ref{Prop2}, and formally dividing Equation~\eqref{Formula} by $G(s+\gamma)$ we obtain
\begin{equation}
\label{S2E101}
K_0(s)=1+\frac {(2-s)G(s)} {\alpha \gamma G(s+\gamma )}.
\end{equation}
 The properties  of the kernel $k_0$ are such that the inverse Mellin transform of  $K_0$ is well defined and  equal to $k_0$ {(for instance by Theorem~11.10.1 in~\cite{misra1986transform} and the proof of the uniqueness theorem below)}. Therefore, by  the equation \eqref{S2E101}, the kernel $k_0$ is given by the inverse Mellin transform of $1+\frac {(2-s)G(s)} {\alpha \gamma G(s+\gamma )}$.  Although in order to prove the uniqueness of $k_0$ it is sufficient to consider the  Mellin transforms $K_0(s)$ and $G(s)$ for real values of $s$, in order to take the inverse Mellin transform, it is necessary to use the values of  $1+\frac {(2-s)G(s)} {\alpha \gamma G(s+\gamma )}$ for  complex values of $s$. {Moreover, we have to ensure that the denominator $G(s+\gamma)$ does not vanish.}

Theorem~\ref{Well_posedness_inverse_pb} below provides an explicit formulation for $k_0$ in terms of the Mellin transform of $g$.
\begin{theorem}
\label{Well_posedness_inverse_pb}
\noindent
Suppose that  {$g$ satisfies~\eqref{S2Emom} and}  is the unique solution of equation \eqref{eq:g} for some given parameters $\alpha >0,$  $\gamma >0,$ and  $k_0$  a non negative measure, compactly supported in $[0, 1]$,  satisfying \eqref{hyp2} 
and \eqref{hyp4}. Let $G(s)$ be  the Mellin transform of the function $g$ as defined in  \eqref{S2EG1}.
Then, there exists $s_0>{2}$ such that 
\begin{eqnarray*}
&&(i)\quad |G(s)|\not =0,\,\,\,\forall s\in \C; \,\,\Re e (s)\in [s_0, s_0+\gamma ],\\
&&(ii) \quad K_0(s)=1+\frac {(2-s)G(s)} {\alpha \gamma G(s+\gamma )},\,\,\,\hbox{for}\,\,\,\Re e(s)=s_0\\
&&(iii) \quad k_0(x)=\frac {1} {2i\pi }\int\limits_{ \Re e(s)=s_0 }\!\!x^{-s}\left( 1+\frac {(2-s)G(s)} {\alpha \gamma G(s+\gamma )}\right)ds.
\end{eqnarray*}
\end{theorem}
{\begin{remark}
\label{}
{Since under our assumptions the function $K_0 $ is only bounded, the integral in point (iii) is understood  in the  sense of distributions, as the second derivative of the inverse Mellin transform of $\f{K_0(s)}{s^2}$ (as it is done in theorem~11.10.1 in~\cite{misra1986transform}.)}
\end{remark}}

Section \ref{Reconstruction} is dedicated to the proof of Theorem~\ref{Well_posedness_inverse_pb}. 

\section{Proof of Theorem~\ref{thm_wp}: uniqueness}
\label{Uniqueness_gamma}
This section is dedicated to the proof of Theorem \ref{thm_wp}. In a first step we prove the uniqueness of  $(\alpha,\gamma)$, thanks to the characterisation of $\gamma$ given by Proposition~\ref{carac_gamma}. In a second step, we prove  the uniqueness of $k_0,$ using well-known properties of the Mellin transform.

In all this section, we assume that the assumption~\eqref{S2Emom} of Theorem~\ref{thm_wp} is satisfied. 
\subsection{Uniqueness of $(\alpha,\gamma).$}
The uniqueness of $(\alpha,\gamma)$  may be proved at least in two different ways, that use  the same property of the solution, namely the behaviour of its moments $G(s),\, s\in [2,+\infty)$, when $s\to\infty.$

The first way relies on the estimates obtained in~\cite{balague:hal-00683148}, Theorem~1.7, that states that there exists a constant $C>0$ and an exponent $p\geq 0$ such that
$$g(x)\sim C x^p e^{-\f{\alpha}{\gamma} x^\gamma},\,\,\hbox{as}\,\,\,x\to \infty$$
from where we deduce
$$\log \left(\f{1}{g}\right) \sim \f{\alpha}{\gamma} x^\gamma, \,\,\hbox{as}\,\,\,x\to \infty$$
which leads to the uniqueness of $(\alpha,\gamma).$

We give here a second  proof of the  uniqueness of $(\alpha,\gamma)$ {because our measure  $k_0$ does not satisfy the hypotheses imposed in ~\cite{balague:hal-00683148}}. It uses estimates on the Mellin transform of $g$,  instead of direct estimates on $g$ itself. However  both proofs strongly rely on the behaviour of high order moments of the function  $g$.

 \begin{proposition}[Necessary condition for $\g$] 
 \label{S3P3}Suppose that $g$ is a function satisfying  \eqref{S2Emom} and solves the  equation\eqref{eq:g} for some  parameters $\gamma >0$,  $\alpha>0$ and  some non negative measure $k_0$, compactly supported in $[0, 1]$,  satisfying \eqref{hyp2} and \eqref{hyp4}.  Let $G$ be the Mellin transform of $g$ for $\Re e(s)\geq 2$. Then:
\\

(i)  The value of the parameter $\alpha $ is uniquely determined by the value of $\gamma $ {(and the function $g$ itself)}.
\\
 
(ii)  Given any constant $R>0$:
\begin{equation}\label{Cond_gamma}
\lim\limits_{s \rightarrow \infty,\, s\in \R^+}   \dfrac{s\,G(s)  }{ G(s+R )} = 
\left\{
\begin{aligned}
&0,\quad\forall R>\gamma \\
&\alpha \gamma ,\,\,\,\hbox{if}\,\,\,R=\gamma \\
&\infty,\quad \forall R\in (0, \gamma )
\end{aligned}
  \right.
 \end{equation}
\label{carac_gamma}
\end{proposition}

\begin{remark}
In order to take the limits in  \eqref{Cond_gamma} we need  $G(s)$ to be defined for all $s$ real  large enough. We  need  then the function $g$ to satisfy condition \eqref{S2Emom}.
\end{remark}

The proof of Proposition  \ref{S3P3} follows  immediately from the two following lemmas.

\begin{lemma}
\label{S3L901}
Under the same hypothesis as in Proposition \ref{S3P3}, the value of the parameter $\alpha $ is uniquely determined by the value of $\gamma $ and
$$
\lim\limits_{s \rightarrow \infty,\, s\in \R}   \dfrac{sG(s)  }{ G(s+\gamma  )} =\alpha \gamma. 
$$
\end{lemma}
\begin{proof}[Of Lemma \ref{S3L901}.] {As seen in} Proposition \ref{Prop2}, the function $G$  is analytic in the domain $$\left\{s, \Re e(s)\geq 2\right\}.$$ Since $g(x)\ge 0$, it follows that $G(s)>0$ for all $s\ge 2$. By condition \ref{hyp2}  on the kernel $k_0$, for all $s>2$, $K_0(s)<K_0(2)=1$. It is then possible to divide both terms of equation \eqref{Formula} by $\gamma G(s+\gamma )(K_0(s)-1)$ to obtain
\begin{equation}
\label{S3Ealpha}
\frac {(2-s)G(s)} {\gamma G(s+\gamma )(K_0(s)-1)}=\alpha  ,\,\,\,\forall s>2.
\end{equation}
Choosing  any $s>2$, this equation uniquely  determines the value of $\alpha $ in terms of the value of $\gamma $.

Lemma~\ref{meth_laplace} in Appendix~\ref{app:K0} proves
  \begin{equation}
  \label{S3E67}
  \lim\limits_{s \rightarrow \infty} K_0(s) = 0. 
 \end{equation}
Combining   \eqref{Formula}, \eqref{S3Ealpha} and \eqref{S3E67}  leads to:
 \begin{equation}
  \lim\limits_{s \rightarrow \infty}   \dfrac{(s-2)G(s) }{ G(s+\g)} = \alpha \gamma, 
 \end{equation}
which ends the proof of~Lemma \ref{S3L901}.
\end{proof}

 \begin{lemma}
 \label{sufficient_condition_gamma}Under the same hypothesis as in Proposition \ref{S3P3},
 \begin{equation}\label{Cond_gamma0}
\lim\limits_{s \rightarrow \infty,\, s\in \R}   \dfrac{sG(s)  }{ G(s+R )} = 
\left\{
\begin{aligned}
&0,&\qquad\forall R>\gamma, \\
&\infty,&\qquad \forall R\in (0, \gamma ).
\end{aligned}
  \right.
 \end{equation}
 \end{lemma}

 \begin{proof}
We first obtain some information on  the asymptotic behaviour of  $G(s)$ for $s\in \R$ and  $s\to \infty$.

\noindent{\bf First Step. On the asymptotic behaviour of $G$ as $s\to \infty$, $s\in \R$.} The idea is to use that  for $s>2$, we have $K_0(s)<1$ and  equation \eqref{Formula} is then equivalent to \eqref{Formula20}, \eqref{Formula22}. Since the function $\Phi $ satisfies
$$
\lim _{ s\to \infty } \frac {A \Phi (s)} {s}=1,\,\,\,A=\alpha \gamma,
$$
the equation  \eqref{Formula} may be considered as being close,  for $s$ large,  to the following equation:  
\begin{equation}
\label{S3E72}
  A G(s+\g) = s G(s), \qquad { s\geq 2}, \qquad G(2) = \rho.
\end{equation}
This is a small variation of the functional equation that defines the Gamma function, where $A=\gamma =1$. {It follows   by Wielandt's theorem (cf. \cite{Rem})} that it has a unique analytical solution $\Gamma_{A,\g}$, given, for all $s>2,$ by
 \begin{equation}\label{Gamma_a,g}
 \Gamma_{A,\g}(s) = c\left(\dfrac{\g}{A}\right)^{\frac{s}{\g}}\Gamma\left( \dfrac{s}{\g}\right),
 \end{equation}
where $\Gamma $ is the Gamma function and  $c$ is a constant uniquely determined by the condition $G(2) = \rho$. 
The  asymptotic behaviour of  $\Gamma_{A,\g}(s)$ as $s\to \infty$ is obtained using Stirling's formula:
\begin{eqnarray}
 \Gamma_{A,\g}(s) &\sim &c\sqrt{\frac {2\pi \gamma } {s} }\left(\dfrac{\g}{A}\right)^{\frac{s}{\g}} 
 \left( \frac {s} {e\gamma }\right)^{\frac {s} {\gamma }}\left( 1+\mathcal O\left(\frac {1} {s}\right)\right), \,\,s\to \infty \nonumber \\
&=& c\sqrt{2\pi \gamma }\,s^{-\frac {1} {2}}e^{\frac {s} {\gamma }(\log s -1 -\log A) }\left( 1+\mathcal O\left(\frac {1} {s}\right)\right).
\label{S3E541}
\end{eqnarray}
We define now:
\begin{equation}\label{defC}
 C(s) = \dfrac{G(s)}{\Gamma_{A,\g}(s)},
 \end{equation}
Since $G$ satisfies  \eqref{Formula20} and  $\Gamma_{A,\g}$ solves \eqref{Gamma_a,g} we deduce that for all  $s\ge 2$:
\begin{eqnarray}
&&C(s+ \g) =  \frac {G(s+\gamma) } {\Gamma_{A,\g}(s+\gamma )}=\frac {A\Phi (s) G(s)} {s\Gamma_{A,\g}(s)}\nonumber\\
&&\hskip 1.55cm = \Psi (s) C(s)\label{S3EqC}\\
&&\hbox{where}:\,\,\ \Psi (s)=\frac {A\Phi (s)} {s},\,\,\,
\Phi (s)=\frac {s-2} {A(1-K_0(s))}.\nonumber
\end{eqnarray}
Lemma \ref{meth_laplace} (see Appendix~\ref{app:K0}) proves that  $K_0(s)=\frac {k_0(1)} {s}+o\left( \frac {1} {s}\right)$ as $s\to \infty$ 
and then,
\begin{eqnarray}
\Psi (s)&=&\frac {s-2} {s(1-K_0(s))}=1+\frac {k_0(1)-2} {s}+\theta (s),\nonumber \\
 \theta (s)&\in & {\cal C}([2, \infty)),\quad \theta(s)=o\left(\frac {1} {s} \right),\quad s\to \infty.\label{feC}
\end{eqnarray}

We deduce from \eqref{S3EqC}\ that, for all $q \ge 2$ fixed and all  $N \in \mathbb{N}^*$:
\begin{equation*}
\begin{aligned}
 C\left(q + N\g\right) &= C\left(q+ (N-1)\g\right) \Psi\left(q+ (N-1)\g\right) \\
 &= C\left(q+ (N-2)\g\right)  \Psi\left(q+ (N-2)\g\right) \Psi\left(q+ (N-1)\g\right) \\
& = \dots \dots \dots \dots\\
&= C\left(q\right)  \Psi \left(q\right)\Psi \left(q+ \g\right)  \dots 
\Psi\left(q+ (N-2)\g\right) \Psi\left(q+ (N-1)\g\right),
 \end{aligned}
\end{equation*}
and then
\begin{equation}
\label{S3E120}
 C\left(q + N\g\right) =C(q) \displaystyle\prod_{k=0}^{N-1} \Psi\left(q + k \g \right).
\end{equation}

Given now any $s>2$, there exists a unique real number $\rho _s\in (0, \gamma )$ and a unique integer $k_s$ such that 
\begin{equation}
\label{S3Erhos}
s=2+\rho _s+\gamma k_s.
\end{equation}
 We may write
\begin{equation}
\label{S3E121}
 C(s)=C(2+\rho _s+\gamma k_s)=C(2+\rho _s)\prod_{k=0}^{k_s-1} \Psi\left(2+\rho _s+\gamma k\right).
\end{equation}
Similarly, for any $s>2$ and $\tilde \gamma >0$, there exists a unique real number $\tilde \rho _s\in (0, \gamma )$ and a unique integer $\tilde k_s$ such that 
\begin{equation}
\label{S3Erhotildes}
s+\tilde \gamma =2+\tilde \rho _s+\gamma \tilde k_s
\end{equation}
 and
\begin{equation}
\label{S3E122}
 C(s+\tilde  \gamma )=C(2+\tilde\rho _s+ \gamma \tilde k_s)=C(2+\tilde\rho _s)\prod_{k=0}^{\tilde k_s-1} \Psi\left(2+\tilde\rho _s+ \gamma  k\right).
\end{equation}
We wish  to estimate now the products in the right hand side terms of \eqref{S3E121} and  \eqref{S3E122}. To this end  we notice that:
\begin{eqnarray*}
&&\log \left(\prod_{k=0}^{N-1} \Psi\left(q + k \g \right)\right)=\sum_{ k=0 } ^{N-1}\log\left(\Psi\left(q + k \g \right) \right)
=\sum_{ k=0 } ^{N-1}\log\left(1+\frac {k_0(1)-2} {q+k\gamma }+\theta (q+k\gamma) \right)\\
&&=
\sum_{ k=0 } ^{N-1}\left(\frac {k_0(1)-2} {q+k\gamma } \right)+\sum_{ k=0 } ^{N-1}\left( 
\log\left(1+\frac {k_0(1)-2} {q+k\gamma }+\theta (q+k\gamma) \right)-\left(\frac {{k_0(1)-2}} {q+k\gamma } \right)\right)\\
&&= {\frac {(k_0(1)-2)}{\gamma} \sum\limits_{k=0}^{N-1} \f{1}{\f{q}{\gamma} +k}+W_q(N)}.
\end{eqnarray*}
{We estimate the first term thanks to the asymptotic properties of the digamma function, that we denote $\psi$. Since $\psi(z+1)=\psi(z)+\f{1}{z}$ we get, a}s $N\to \infty$:
\begin{eqnarray*}
{ \sum\limits_{k=0}^{N-1} \f{1}{\f{q}{\gamma} +k}=\psi(N+\f{q}{\gamma}) }-\psi(\f{q}{\gamma})=
 \log (N)+\omega_q (N),
\end{eqnarray*}
where  $\omega_q (s)$ is a continuous function on $r>0$ such that, for any $R>2$ there exists $C_R>0$ for which:
$$
|\omega _q(s)|\le C_R,\,\,\,\forall s>2,\,\,\,\forall q\in  (2, R).
$$
{We now estimate the term $W_q(N)$}. For all $\varepsilon >0$ and all $R>2$, there exists $M_\varepsilon >0$ such that for all $q\in (2, R)$:
\begin{eqnarray}
|W_q(N)|&\le& \sum_{ k=0 } ^{M_\varepsilon }\left| 
\log\left(1+\frac {k_0(1)-2} {q+k\gamma }+\theta (q+k\gamma) \right)-\left(\frac {{k_0(1)-2}} {q+k\gamma } \right)\right|
+\varepsilon \sum_{ k=M_\varepsilon  } ^{N-1}\frac {|k_0(1)-2|} {q+k\gamma } \nonumber \\
&\le & C _{ M_\varepsilon  }(q)+\varepsilon \sum_{ k=0} ^{N-1}\frac {|k_0(1)-2|} {q+k\gamma }\nonumber \\
&=&C_{M_\varepsilon }(q)+\varepsilon  \frac {|k_0(1)-2|} {\gamma }\left(\psi \left({N}+\frac {q} {\gamma } \right)-\psi \left(\frac {q} {\gamma } \right)\right)\nonumber \\
&\le &C_{M_\varepsilon }(q)+\varepsilon \frac {|k_0(1)-2|} {\gamma }\left(\log (N)+\omega _q(N)\right).
\label{S3EW}
\end{eqnarray}
We deduce that 
\begin{eqnarray*}
&&\prod_{k=0}^{N-1} \Psi\left(q + k \g \right){\leq} N^{\frac { k_0(1)-2} {\gamma }}e^{\Omega (q, N)},\\
&&\Omega (q, N)=\frac {k_0(1)-2} {\gamma }\omega _q(N)+W{_q}(N).
\end{eqnarray*}
Then by \eqref{S3EW}, for all $\varepsilon >0$ and  for any $R>2$, there exists two positive constants $C _{ 1, R }$ and $C _{ 2, R }$ such that for all $q\in (2, R)$ and all $N$ sufficiently large:
\begin{eqnarray}
\label{S3E421}
C _{ 1, R }\, N^{\frac {(k_0(1)-2)-\varepsilon |k_0(1)-2| } {\gamma }}\le \prod_{k=0}^{N-1} \Psi\left(q + k \g \right)\le C _{ 2, R }\, N^{\frac {(k_0(1)-2)+\varepsilon |k_0(1)-2| } {\gamma }}.
\end{eqnarray}

Consider now any $\tilde \gamma >0$. By definition of the function $C$(s) we have:
\begin{equation*}
\dfrac{s G(s)}{G(s+ \tilde{\g})}  = 
\dfrac{s\Gamma_{A,\g}(s) C(s)}{\Gamma_{A,\g}(s+\tilde{\g})C(s+ \tilde{\g})}. 
 \end{equation*}
Since for all $s\ge 2$ and $\tilde \gamma >0$ the real numbers $\rho _s$ and $\tilde \rho _s$ defined in \eqref{S3Erhos} and \eqref{S3Erhotildes}  are in $(0, \gamma )$, we have by  \eqref{S3E421}:
\begin{equation*}
\dfrac{C(s)}{C(s+ \tilde{\g})} =  \frac {C(2+\rho _s)} {C(2+\tilde \rho _s)}
\frac {\prod_{k=0}^{k_s-1} \Psi\left(2+\rho _s+\gamma k\right)} {\prod_{k=0}^{\tilde k_s-1} \Psi\left(2+\tilde \rho _s+\gamma  k\right)}\\
\le C \frac {C(2+\rho _s)} {C(2+\tilde \rho _s)}k_s^{\frac {(k_0(1)-2)+\varepsilon |k_0(1)-2| } {\gamma }}
 {\tilde k_s}^{\frac {-(k_0(1)-2)+\varepsilon |k_0(1)-2| } {\gamma }}.
 \end{equation*}
for some constant $C>0$. By definition $k_s< s/\gamma $ and if $s\to \infty$, since $\tilde \rho _s\in (0, \gamma )$ it is easy to check that $\tilde k_s >(s+\tilde \gamma -2-\gamma )/\gamma >s/2\gamma $ for $s$ large enough. On the other hand, notice that for all $s>2$ and $\tilde \gamma >0$, we have by definition $2+\rho _s\in (2, 2+\gamma )$ and  $2+\tilde \rho _s\in (2, 2+\gamma )$. Since the function $C$ is continuous on $[2, \infty)$ and strictly positive, there exists a positive constant $C$ such that, for all $s>2$ and $\tilde \gamma >0$:
$$
 \frac {C(2+\rho _s)} {C(2+\tilde \rho _s)}\le C.
$$
It then follows that for some constant $C>0$:
 \begin{equation}
 \label{S3E97}
\dfrac{C(s)}{C(s+ \tilde{\g})} 
\le Cs^{\frac {2\varepsilon |k_0(1)-2| } {\gamma }},\,\,\,\forall s>2.
 \end{equation}
A similar argument shows that for some constant $C'>0$:
 \begin{equation}
\label {S3E98}
\dfrac{C(s)}{C(s+ \tilde{\g})} 
\ge C's^{\frac {-2\varepsilon |k_0(1)-2| } {\gamma }},\,\,\,\forall s>2.
 \end{equation}
We deduce from \eqref{S3E541}:
\begin{eqnarray}
\label{S3E896}
\dfrac{s\Gamma_{A,\g}(s)}{\Gamma_{A,\g}(s+\tilde{\g}))} &=&
\frac {s^{{-\frac {1} {2}}}s e^{\frac {s} {\gamma }(\log s -1 -\log A) }}
 {(s+\tilde \gamma )^{{-\frac {1} {2}}} e^{\frac {s+\tilde \gamma } {\gamma }(\log (s+\tilde \gamma ) -1 -\log A) }}\left(1+\mathcal O\left( \frac {1} {s}\right) \right)\,\,\hbox{as}\,\,s\to \infty.
 \end{eqnarray}
Simple calculus gives:
{
\begin{eqnarray}
&&-\f{1}{2}\log(s+\tilde \gamma) +\f{s+\tilde \gamma}{\gamma}(\log(s+\tilde \gamma) -1 - \log A)=\nonumber \\
&&\hskip 1cm =-\f{1}{2}\log s+\mathcal O (\f{1}{s})  +\f{s}{\gamma}(1+\f{\tilde \gamma}{s})(\log s+\f{\tilde \gamma}{s}+\mathcal O(\f{1}{s^2}) -1 - \log A)\nonumber\\
&&\hskip 1cm=-\f{1}{2}\log s+ \f{s}{\gamma}(\log s+\f{\tilde \gamma}{s} -1 - \log A) +\f{\tilde \gamma}{\gamma}(\log s -1 - \log A) +\mathcal O\left(\f{1}{s}\right)\nonumber
\,\,\hbox{as}\,\,s\to \infty
\end {eqnarray}
and
\begin{eqnarray}
\f{1}{2}\log(s)+\f{s}{\gamma}(\log s - 1 -\log A) &=&\f{1}{2}\log(s)+\f{s}{\gamma}(\log s - 1 -\log A)\,\,\,\,\hbox{as}\,\,s\to \infty, \nonumber
 \end{eqnarray}
from where
\begin{eqnarray*}
\frac {s^{-\frac {1} {2}}s e^{\frac {s} {\gamma }(\log s -1 -\log A) }}
{(s+\tilde \gamma )^{-\frac {1} {2}} e^{\frac {s+\tilde \gamma } {\gamma }(\log (s+\tilde \gamma ) -1 -\log A) }}
&=&\f{s e^{\mathcal O (\f{1}{s})}}{e^{ \f{\tilde \gamma}{\gamma}+\mathcal O(\f{1}{s} ) +\f{\tilde \gamma}{\gamma}(\log s -1 - \log A)}} \nonumber \\
&=& {s e^{\mathcal O (\f{1}{s}) - \f{\tilde \gamma}{\gamma}  -\f{\tilde \gamma}{\gamma}(\log s -1 - \log A)}},\,\,\, s\to \infty 
\end{eqnarray*}
and:
 \begin{eqnarray}
 \dfrac{s\Gamma_{A,\g}(s)}{\Gamma_{A,\g}(s+\tilde{\g}))} &=&{A^{\frac {\tilde \gamma } {\gamma } }}
  s^{1-\frac {\tilde \gamma } {\gamma }} \left(1+\mathcal O\left( \frac {1} {s}\right) \right)\,\,\hbox{as}\,\,s\to \infty.  \label{S3E897}
 \end{eqnarray}
 }
From \eqref{defC}, \eqref{S3E97}, \eqref{S3E98} 
and \eqref{S3E897} we deduce that for any $\tilde \gamma >0$ and any $\varepsilon >0$ small,   there exists two positive constants $C_1$ and $C_2$ such that  for all $s$  sufficiently large:
\begin{equation}
\label{S3bounds}
C_1 s^{1-\frac {\tilde \gamma } {\gamma }-\frac {2\varepsilon |k_0(1)-2| } {\gamma }}\le \frac {sG(s)} {G(s+\tilde \gamma )} \le C_2 s^{1-\frac {\tilde \gamma } {\gamma }+\frac {2\varepsilon |k_0(1)-2| } {\gamma }}.
\end{equation}

\noindent{\bf Second Step. End of the proof of Lemma \ref{sufficient_condition_gamma}.}

We may conclude now  the proof of Lemma \ref{sufficient_condition_gamma}  using the estimates \eqref{S3bounds} as follows. 
Suppose  that $\tilde \gamma >\gamma $. Then, we choose $\varepsilon >0$ in \eqref{S3bounds}  small enough in order to have:
$$
1-\frac {\tilde \gamma } {\gamma }+\frac {2\varepsilon |k_0(1)-2| } {\gamma }<0.
$$
It follows from the upper estimate in \eqref{S3bounds} that
$$
\lim _{ s\to \infty }\frac {sG(s)} {G(s+\tilde \gamma )}=0.
$$
If,  on the other hand   $\tilde \gamma <\gamma $ we choose  $\varepsilon >0$ \eqref{S3bounds} small enough in order to have
$$
1-\frac {\tilde \gamma } {\gamma }-\frac {2\varepsilon |k_0(1)-2| } {\gamma }>0.
$$
As a consequence of the lower estimate in  \eqref{S3bounds} we deduce that 
$$
\lim _{ s\to \infty }\frac {sG(s)} {G(s+\tilde \gamma )}=\infty.
$$
This concludes the proof of  \eqref{Cond_gamma0} and of Lemma \ref{sufficient_condition_gamma}.
\end{proof}

\begin{proof} [\textbf{Of Proposition \ref{S3P3}}] Property (i) and property (ii) for $R=\gamma $ follow  \ref{S3L901}. Property (ii) for $R>0$, $R\not =\gamma $ follows from Lemma \ref{sufficient_condition_gamma}.
\end{proof} 
 
\begin{proof} \textbf {[Of Theorem \ref{S2T1}]}
From the formula~\eqref{Formula}, and since by definition $G(s+\gamma)$ is strictly positive for $s\in {[1}-\gamma,+\infty)$ we can divide by $G(s+\gamma)$ and obtain formula~\eqref{S2E101}, {\it i.e.}
$$K_0(s):=1+\f{G(s)(2-s)}{\alpha \gamma  G(s+\gamma)},\qquad 1\leq s <+\infty.$$
This determines uniquely the Mellin transform $K_0(s)$  of the measure $k_0$ for all $s\ge 1$. We want to use now a uniqueness theorem for the Laplace transforms of measures.
The Laplace transform of a measure $\mu $ is defined as
$$
\mathcal L(\mu )(s )=\int _0^\infty e^{-sy}d\mu (y).
$$
\textcolor{black}{We claim now that}
$K_0(s)$ is the Laplace transform of the non negative measure $d\mu_0(y)$ defined on $(0, \infty)$  by
\textcolor{black}{\begin{equation}\label{pushforward}
d \mu_0(x) = e^{2x} \; T \# \;  xdk_0(x), \qquad \text{with } T(x)= -\log(x).
\end{equation} 
We recall that the pushforward measure $\mu$ of a measure $\nu$ by the function $T$, that we denote $\mu= T \# \nu,$ is such that for all $\psi \in L^1(d\mu)$, we have $\dst\int (\psi \circ T) d \nu= \dst\int \psi d \mu$
.
Indeed, using \eqref{pushforward}, 
\begin{eqnarray}
K_0(s)&=&\int _0^1 x^{s-1}dk_0(x)=\int _0^1 x^{s-2}xdk_0(x)= \int _0^1  e^{(s-2)\log x}xdk_0(x) \nonumber\\
&=&\int _0^\infty e^{-(s-2)x}\; T\# x dk_0(x) = \int _0^\infty e^{-sx}\;d\mu_0(x) \nonumber\\
&=& \mathcal L(\mu_0 )(s ).
\end{eqnarray}
}

Suppose now that  $k_0^{(1)}$ and $k_0^{(2)}$ are two measures, satisfying \eqref{hyp2} and ~\eqref{hyp4}, such that the same function $g$, satisfying \eqref{S2Emom}  solves the equation  \eqref{eq:g}  for $k_0^{(1)}$ and $k_0^{(2)}$. Then, by Proposition \ref{Prop2} their  Mellin transforms,  
$K_0^{(1)}$ and $K_0^{(2)}$ respectively, satisfy \eqref{S2E101} and then $K_0^{(1)}(s)=K_0^{(2)}(s)$ for all $s\ge 1$. The Laplace transforms of the corresponding non negative measures $\mu_0^{(1)}$ and $\mu_0^{(2)}$ are then equal for all  $s\ge 1$. It follows that  $\mu_0^{(1)}=\mu_0^{(2)}$ (cf. Theorem 8.4. in~\cite{bhattacharya}),  and then $k_0^{(1)}=k_0^{(2)}$.
\end{proof}

\begin{remark}
The method used to prove the uniqueness of the triplet is based on the fact that under hypotheses on $k_0$, 
 the Mellin Transform $K_0(s)$ for $s \in \mathbb{R}$ goes to 0 as $s$ goes to $+\infty$ with a convergence rate of at least $1/s$. {This is a point which could be relaxed,  by assuming{~\eqref{hyp:DG} with a power $\nu>-1$} as in~\cite{balague:hal-00683148}: the convergence could then be slowlier {and Lemma~\ref{meth_laplace} needs to be adapted,  but it }would still be sufficient.}\\\\
\end{remark}
 \section{Reconstruction of $k_0$}
 \label{Reconstruction}
This section is dedicated to the proof of Theorem \ref{Well_posedness_inverse_pb}. To reconstruct the kernel $k_0$ from the function $g$ and the parameters $\alpha$ and $\gamma,$ we want to do an inverse Mellin transform of the functional of $G(s)$ given in~\eqref{S2E101}. This requires to integrate this functional on a line $s_0 +i\R$ with $s_0>2,$ see Formula $(iii)$ in Theorem~\ref{Well_posedness_inverse_pb}. Since it is necessary to divide by the function  $G(s+\gamma )$, we have to prove that $G(s+\gamma)$ does not vanish on the line $s_0+\gamma +i\R$. We already know that this is true  for real values $s\ge  1$, but is not known for $s\in \C$. 

To do so, {\color{black}Proposition}~\ref{S4L40} defines an explicit expression for a solution $\tilde G(s)$ to~\eqref{Formula}, from  which it is easy to deduce that $\tilde G$ does not vanish. This  expression~\eqref{G:candidate} is obtained using {the Cauchy integral}, see Lemma~\ref{S4L5}. However, since the equation~\eqref{Formula} may admit several solutions, we need to prove that $\tilde G(s)$ given by~\eqref{G:candidate} is equal to the Mellin transform $G(s)$ of the solution $g$ to~\eqref{eq:g}. This is done by studying the inverse Mellin transform of $\tilde G$ (Lemma~\ref{S4L1}) and then using a uniqueness result of solutions to the equation \eqref{eq:g} (Theorem~\ref{inv_mellin}).

As a first, the following lemma obtains a solution of  problem~\eqref{Formula}.
\begin{prop}
\label{S4L40}[Solution  to Problem~\eqref{Formula}.]
{Let $k_0$ satisfy the hypotheses~\eqref{hyp2},~\eqref{hyp4}.} For any $s_0>2$ and $\varepsilon >0$ such that $s_0-\frac {\gamma \varepsilon } {\pi }>2$, define the complex valued function $\tilde{G}$ for $\Re e(s) \in (s_0,s_0+\gamma)$ as 
\begin{eqnarray}
\label{G:candidate}
\widetilde{G}(s) = \exp \Bigg(- \dfrac{1}{\gamma }\dst\int_{\Re e(\sigma) =s_0}
\log\big(\Phi( \sigma)\big)\Bigg\{ \dfrac{1}{1-e^{2 i \pi \frac{s-\sigma }{\gamma}}}  -\dfrac{1}{2+e^{2 i \pi \frac{s_0-\sigma}{\gamma}}}\Bigg\}\;     d\sigma  \Bigg),\; 
\end{eqnarray}
where $\log(z)$ denotes  the following determination of the logarithm:
\begin{equation}
\label{S4Elog}
\log (z)=\log|z|+i \arg z, \,\,\, \arg z\in [0, 2\pi )
\end{equation}
and the function $\Phi (s)$ is defined by \eqref{Formula22} in Remark \ref{S2RFi}.
{\color{black} Then, the function $\widetilde{G}$ may be extended analytically to $\C$ by the formula
\begin{equation}\label{def_G_prolong}
\widetilde{G}(s) = \exp \Bigg(- \dfrac{1}{\gamma }\dst\int_{\Re e(\sigma) =s_0}
\log\big(\Phi( \sigma)\big)\Bigg\{ \dfrac{1}{1-e^{2 i \pi \frac{s-\sigma }{\gamma}}}  -\dfrac{1}{2+e^{2 i \pi \frac{s_0-\sigma}{\gamma}}}\Bigg\}\;     d\sigma  \Bigg)e^{\f{k}{2}\log \left(\Phi(s_0 +i\Im m(s))\right)},
\end{equation}
with $\Re e(s) \in (s_0+k\gamma,s_0+(k+1)\gamma)$ for $k\in \Z.$
Moreover, $\widetilde G$ solves ~\eqref{Formula} in $\C.$ 
}
 \end{prop}
{\color{black} Formula~\eqref{def_G_prolong} is only valid for $\Re e(s) \notin s_0+ \gamma \Z$, but the proof below shows that it can be continuously extended to $\C$.}

\begin{proof} 
We first obtain an explicit solution in terms of a new variable $\zeta $ defined as follows:
 \begin{equation}
 \label{changevarbles}
 \zeta = T(s)=e^{2 i \pi \frac{s-s_0}{\gamma}}.
 \end{equation}
The conformal mapping  $T$ transforms the complex plane $\mathbb{C}$ into the Riemann surface denoted $S$, associated to the logarithmic function. Every point $\zeta$ of each sheet of this surface is characterized uniquely  by its modulus $|\zeta |$ and the  determination of its argument $\theta$,   with $\theta \in (2k\pi, 2(k+1)\pi)$ for the $k$th sheet. 
  The function $\widetilde \varphi $ defined as
\begin{equation}
\label{S4Etildephi}
\forall \zeta \in S:\,\,\, \widetilde \varphi(\zeta )= \widetilde \varphi(T(s) )=\Phi (s)
\end{equation}
 is meromorphic in $S$.
 Any  vertical infinite strip in $\mathbb{C}$ of the form {$\Re e  (\frac {s-s_0} {\gamma } )\in (-\f{\eps}{2\pi}, 1+\f{\eps}{2\pi})$}  for some ${\eps>0}$ ,  is transformed by $T$ into a portion $ {D(\eps)}$  of $S $ defined by 
\begin{equation}\label{def:Deps}
D{(\eps)}=\left\{\zeta \in S ; \quad \zeta =re^{i\theta}, \quad \theta \in {(-\eps,2\pi +\eps )},\quad {r>0} \right\}.
\end{equation}
In particular, the map $T$ is a bijection between the strip $\{z \in \mathbb{C}\; |\;s_0  < {\Re}e(z) < s_0 + \gamma\} $ and its image $T\bigl(\{z \in \mathbb{C}\; |\;s_0  < \Re e(z) < s_0 + \gamma\} \bigr)$ where: 
 $$T\left( \{z \in \mathbb{C}\; |\;s_0  < \Re e(z) < s_0 + \gamma\}\right)=\left\{\zeta \in \C;\quad \zeta =|\zeta |e^{i\theta}, \,\,\,\theta\in {(}0, 2\pi )  \right\}$$   
(see Figure \ref{WH}). The inverse of $T$ on  $T\big(\{z \in \mathbb{C}\; |\;s_0  < \Re e(z) < s_0 + \gamma\} \big)$ is then  given by:
$$
T^{-1}(\zeta )=s_0+\frac {\gamma } {2i\pi }\log (\zeta ).
$$
Notice also that:
  \begin{equation}
  \left\{
  \begin{aligned}
  & \hbox{if}\quad\Re e(s) \rightarrow (s_0)^+\,\,&\hbox{then}\,\, \arg(\zeta) \rightarrow 0^+,\\
    & \hbox{if} \quad \Re e(s) \rightarrow (s_0+ \gamma)^- \,\,&\hbox{then}\,\,  \arg(\zeta) \rightarrow (2\pi)^-,\\
  \end{aligned}
  \right.
 \end{equation}
where the function $\arg(\cdot)$ is determined as  in \eqref{S4Elog}, i.e.  $\arg(z)\in [0, 2\pi )$. \\
We look  for a solution $\widetilde G$ of \eqref{Formula} of the form:
\begin{eqnarray}
\label{CV}
&&\widetilde G(s)=F({T}(s)).
 \end{eqnarray}
If $\widetilde G$ has to be  analytic in { $\Re e  (s)\in (s_0-\f{\eps\gamma}{2\pi}, s_0+\gamma+\f{\eps\gamma}{2\pi})$} and must satisfy  the equation ~\eqref{Formula} in that strip, then the function $F $ should be analytic on $D{(\eps)}$ and satisfy:
\begin{equation}
\label{eq:fonc_mult}
F(r - i 0 )  =  \varphi(r) F(r+i0), \qquad \forall r >0,
\end{equation}

\begin{figure}[htbp]
 \begin{center}
  \includegraphics[width=0.8\textwidth]{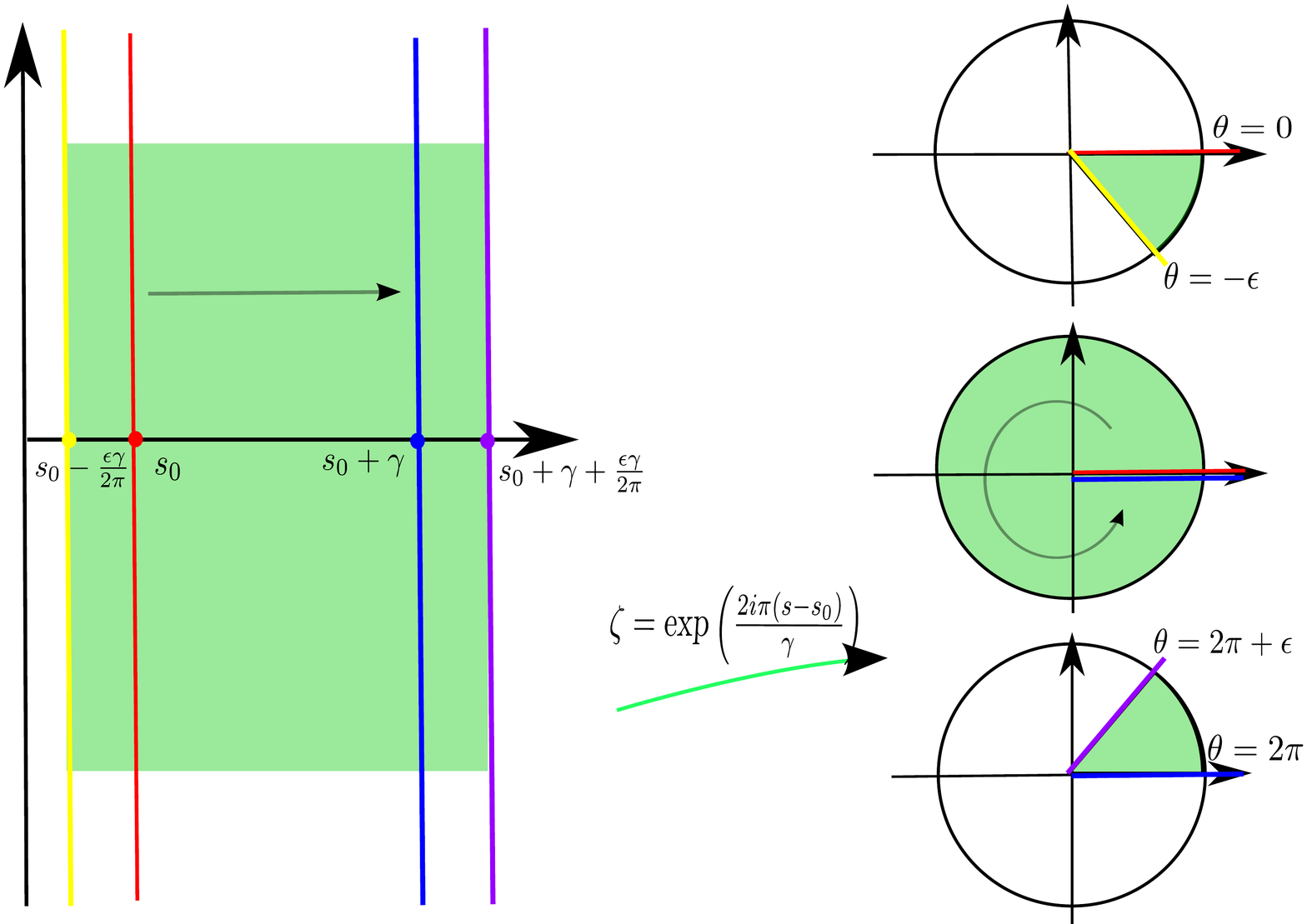}
  \caption{\label{WH}Holomorphic change of coordinates sends the green strip $s_0 - \dfrac{\eps\g}{2\pi} < \Re e(s) < s_0 + \g + \dfrac{\eps\g}{2\pi}$  (Left)
  into  the green subdomain $D(\eps)$ (Right) of the Riemann surface associated to the logarithmic function.}
  \end{center}
 \end{figure}

where we define, for all $r>0$:
\begin{equation}\label{def:Fi0}
\begin{aligned}
 F(r+ i0 ) :=   \lim\limits_{\eta \rightarrow 0^+ } F(r e^{i\eta}), \quad
  F(r - i0 ) :=   \lim\limits_{\eta \rightarrow 0^+ } F(r e^{i(2\pi-\eta)}), \quad
  {\varphi(r):= \lim\limits_{\eta\to 0^+} \tilde{\varphi}(re^{i\eta})}.
  \end{aligned}
\end{equation}

In order to make appear the multiplicative functional equation \eqref{eq:fonc_mult} as a typical Carleman equation \cite{Zakharov98} (additive), we  look for  the function $F(\zeta )$ of the form 
 \begin{equation}
F(\zeta)=e^{P(\zeta )}
 \end{equation}
 where the function $P(\zeta )$  is such that:
 \begin{equation}\label{eq:fonc_add}
P(r-i0) =\log(\varphi(r)) + P(r+i0), \qquad  \forall r>0.
 \end{equation}

 The existence of such a function $P$ with the suitable properties is proved in the following  lemma.
 \begin{lemma} 
 \label{S4L5}{Let $k_0$ satisfy the hypotheses~\eqref{hyp2},~\eqref{hyp4}.} The function  $ P(\zeta)$ defined for all $\zeta \in $
 \comM{$D(0)$}
 by
 \begin{equation}\label{Form_PS}
  P(\zeta)  =- \dfrac{1}{2i\pi}\dst\int_0^{+\infty} \log(\varphi(w))\Big\{ \dfrac{1}{w-\zeta}  -\dfrac{1}{w+1}\Big\}\; dw
 \end{equation}
  \comM{and can analytically {\color{black} be} extended on $S$, its unique analytic continuation} 
satisfying the equation \eqref{eq:fonc_add}. Moreover,
 \comM{
 the analytical continuation of $P$ (which we will denote by $P$ as well) on $S$
 has a simple expression 
 \begin{equation}\label{analytic_cont_P}
  P(\zeta) = P(\tilde{\zeta}) + \dfrac{k}{2} \log(\varphi(|\zeta|)),\quad  \qquad \arg(\zeta) \in (2k\pi,2 (k+1)\pi), \;  k \in \Z,
 \end{equation}
 where $\tilde{\zeta} \in D(0), \; {\color{black}\vert  \zeta\vert=\vert \tilde\zeta \vert},\;\arg(\tilde{\zeta}) \equiv \arg(\zeta) \mod (2\pi).
 $}
\end{lemma}

\begin{remark}{In the definition of $P(\zeta)$, we could choose as well a fraction $\f{1}{w+a}$ instead of $\f{1}{w+1}$, for any $a>0,$} \comM{or any function $h(w)$ rendering \eqref{Form_PS} convergent.}
{The difference between two such} \comM{definitions of $P(\zeta)$} {would be given by a converging integral, namely
$$\f{1}{2i\pi} \dst\int_0^{+\infty} \log(\varphi(w))\Big\{ \comM{h(w)} -\dfrac{1}{w+1}\Big\}\; dw,$$
which is independent of $\zeta.$ This would lead to a multiplicative constant for the definition of $F,$ hence of $\tilde G.$ \comM{The function $h$} could be determined by the normalisation $\tilde G(2)=\rho$ as in~\eqref{eq:gBIS}; however to keep it simple, we stick here to the choice \comM{$h(w) =\dfrac{1}{w+a}$  and} $a={\color{black}1}$.
}\end{remark}

\begin{proof}[\textbf{Of Lemma \ref{S4L5}}]
From the hypothesis on $k_0$ we have,  for all $s\in \C$ such that $\Re e (s)>2$:
$$
|K_0(s)|\le \int _0^1|x^{s-1}|dk_0(x)= \int _0^1x^{\Re e  (s)-1}|dk_0(x)<\int _0^1xdk_0(x)=1.
$$
Therefore, the function $\Phi (s)$ is analytic in the domain of $\C$ defined by $\Re e(s)>2$ and $\Phi (s)\not =0$ in that domain.
Suppose now that $s_0>2$. 
By definition,  $T$ is a bijection from the strip
\begin{equation*}
\left\{s\in \C; \Re e (s)\in \left(s_0-\frac {\varepsilon \gamma } {2\pi }, s_0+\frac {\varepsilon \gamma } {2\pi } \right) \right\}.
\end{equation*}
into the piece of the Riemann's surface:
\begin{equation*}
\left\{z\in S; z=|z|e^{i\theta},\,\,\theta\in (-\varepsilon , \varepsilon )  \right\}.
\end{equation*}
Since by definition,  $\tilde \varphi (\zeta )=\Phi (T^{-1}(\zeta ))$,  we deduce that $\tilde \varphi $ is analytic on the domain
$\left\{\zeta \in S; \zeta =|z|e^{i\theta},\,\,\theta \in (-2\varepsilon , 2\varepsilon )\right\}$ and $\tilde \varphi (\zeta )\not =0$ in that domain. The function $\log \tilde \varphi (\zeta )$ is then well defined in that domain. It follows from Lemma \ref{S4Estlog} (in Appendix~\ref{sec:append:3lemmas}) that, for all $\zeta \not \in \R^+$:
\begin{equation*}
\int_0^{+\infty} \left|\log(\varphi(w))\right|\left | \dfrac{1}{w-\zeta}  -\dfrac{1}{w+1}\right|\; dw<\infty
\end{equation*}
and the function $P$ is then well defined and analytic in ${D(0)=}\left\{\zeta \in S; \zeta =|z|e^{i\theta},\,\,\theta\in (0, 2\pi ) \right\}$.\\
The analyticity of $P$ on the domain $D(\varepsilon )$ follows from the analyticity of the function $\log(\varphi(w))$ on $\left\{z\in S; z=|z|e^{i\theta},\,\,\theta\in (-\varepsilon , \varepsilon )  \right\}$, a deformation of the contour of integration from $\R^+$ to rays $e^{i\theta}\R^+$ with $|\theta|< \varepsilon $ and using Lemma \ref{S4Estlog}. 

From  {Sokhotski-Plemelj} formulas (see Appendix {\ref{sec:Sokh}}), applied to {the test function $f(w) =\f{-1}{2i\pi} \frac{(2+r)\log(\varphi(w))}{(w+1)}$} we obtain, for all $r>0$: 
  \begin{equation}
  \begin{aligned}
   P(r + i0) &= -\dfrac{1}{2}\log(\varphi(r)) 
   - \dfrac{1}{2i\pi}P.V.\dst\int_0^{\infty} log(\varphi(w))\Big\{ \dfrac{1}{w-{r}}  -\dfrac{1}{w+1}\Big\}\; dw, \\
   P(r - i0) &= \dfrac{1}{2}\log(\varphi(r)) 
   - \dfrac{1}{2i\pi}P.V.\dst\int_0^{\infty} log(\varphi(w))\Big\{ \dfrac{1}{w-{r}}  -\dfrac{1}{w+1}\Big\}\; dw.   
  \end{aligned}
 \end{equation}
The notation $P.V.$ stands for the usual principal value of Cauchy.\\
If we take now the difference between these two formulas we deduce that, for all $r>0$:
\begin{equation*}
\log(\varphi (r))=P(r-i0)-P(r+i0),
\end{equation*}
{\color{black} from which we deduce, by induction, the formula~\eqref{analytic_cont_P}.}
\end{proof}

\textbf{End of the proof of Proposition~\ref{S4L40}. }
We deduce that the function 
\begin{equation}
F(\zeta )=\exp\left(P(\zeta ) \right) 
\end{equation} given by
\begin{equation}\label{F}
F(\zeta) = \exp \Big(- \dfrac{1}{2i\pi}\dst\int_0^{+\infty} \log(\varphi(w))\Big\{ \dfrac{1}{w-\zeta}  -\dfrac{1}{w+1}\Big\}\; dw \Big), \qquad \comM{\zeta \in D(0)},
\end{equation}
\comM{can be analytically continued in {\color{black} $S$}}
and satisfies  \eqref{eq:fonc_mult}. Using the change of variables \eqref{changevarbles} in  \eqref{F} gives the expression \eqref{G:candidate} for the function $\widetilde G$.
{\color{black} From~\eqref{analytic_cont_P}, we moreover get the two following useful formulae, for $\zeta\in S$ and $\tilde\zeta\in D(0)$, with $\arg (\zeta)\in (2k\pi,2(k+1)\pi)$, $\vert\zeta\vert=\vert\tilde\zeta\vert$ and $\arg(\zeta)-\arg(\tilde \zeta)=2k\pi$:}
{\color{black}\begin{equation}\label{def_F_prolong}
F(\zeta)=\exp (P(\tilde\zeta)) \exp\left(\f{k}{2} \log \left(\varphi(\vert\zeta\vert)\right)\right),
\end{equation}
\begin{eqnarray*}
&&\widetilde{G}(s) = \exp \Bigg(- \dfrac{1}{\gamma }\dst\int_{\Re e(\sigma) =s_0}
\log\big(\Phi( \sigma)\big)\Bigg\{ \dfrac{1}{1-e^{2 i \pi \frac{s-\sigma }{\gamma}}}  -\dfrac{1}{\comM{1}+e^{2 i \pi \frac{s_0-\sigma}{\gamma}}}\Bigg\}\;     d\sigma  \Bigg)\times \\
&&\hskip 9cm \times \exp\left(\f{k}{2}\log \left(\Phi(s_0 +i\Im m(s))\right)\right),
\end{eqnarray*}
with $\Re (s) \in (s_0+k\gamma,s_0+(k+1)\gamma).$}
\end{proof}
\\
 For $s_0-\f{\eps\gamma}{2\pi} >2,$ both functions $\widetilde{G}$ and $G$ 
are analytic on $\left\{ s\in \C; \quad \Re e (s)\in (s_0-\frac {\varepsilon \gamma } {2\pi } , s_0+\frac {(2\pi +\varepsilon)\gamma  } {2\pi }) \right\}$ and satisfy  \eqref{Formula}, 
but nothing guarantees yet  that $G = \tilde{G}$.
We notice for instance that for $\Phi(s) = s$ and $\g =1$, 
the functions 
\begin{equation}
s \mapsto \frac{1}{2} \Gamma(s)  \quad \text{and} \quad s \mapsto \frac{1}{2}\Gamma(s) \Big[ 1+ \sin(2\pi s)\Big] 
\end{equation}
are two distinct solutions~{to the equation $\Gamma(s+1)=s\Gamma(s)$}. The first one never cancels whereas the second one does. Our purpose now is to show that the inverse Mellin transform of $\widetilde G$ exists, belongs to  $ L^1((1+x^{\gamma +1})dx)$ and satisfies \eqref{eq:g}: {we then conclude by uniqueness of solutions to~\eqref{eq:g} in $ L^1((1+x^{\gamma +1})dx)$, see {\it e.g.} Theorem 3.1. in~\cite{MMP05}, and by the properties of the inverse Mellin transform, see Theorem~11.10.1 in~\cite{misra1986transform}.}


\begin{theorem}\label{inv_mellin}
Let $g$ be the solution to the stationary equation \eqref{eq:g} satisfying \eqref{S2Emom}, and  $\widetilde{G}$ the function defined in \eqref{G:candidate} {\color{black} for $s_0>2+\gamma$}. 
Then
\begin{equation}
\label{S4E297}
g(x) = \dfrac{1}{2 i \pi} \dst\int_{\Re e(s)=u} \widetilde G(s)x^{-s}ds,\,\, ,  \forall u  > s_0.
\end{equation}
\end{theorem}
The proof of Theorem \ref{inv_mellin} is done in two steps. We first prove that the inverse Mellin transform of $\widetilde G$, that we denote  $\tilde g$,  is  a function, with suitable integrability properties on $(0, \infty)$. The theorem  then follows using a uniqueness result for the solutions of the equation \eqref{eq:g}.

{The results are based on} the behaviour of $\tilde\varphi$ on the kernel $k_0$, see Lemma~\ref{S4Estlog} in Appendix~\ref{sec:append:3lemmas}).
From this,  we derive the asymptotic behaviour of $F(\zeta)$
as $\vert\zeta\vert$ goes to $0$ or to $\infty$.
\begin{lemma}
\label{S4L1}
 The inverse Mellin transform of $\widetilde G$ defined { in the sense of distributions} as 
  \begin{equation}
  \label{gtilde}
  \widetilde g(x) = \dfrac{1}{2 i \pi} \dst\int_{\Re e(s)=u} \widetilde G(s)x^{-s}ds
 \end{equation}
 for $u \in \left(s_0-\frac {\varepsilon \gamma } {2\pi } , s_0+\frac {(2\pi +\varepsilon)\gamma  } {2\pi }\right)$, 
satisfies:
\begin{equation}
\label{S4Egtilde}
g\in L^1((x+x^{\gamma +1})). 
\end{equation}
 \end{lemma}

\begin{proof}
We first prove that the integral in \eqref{gtilde} is convergent. To this end we first recall that  $\widetilde G$ is analytic in the domain ${\left\{s\in \C; \Re e (s) \in  (s_0-\f{\eps\gamma}{2\pi}, s_0+\f{\eps\gamma}{2\pi} )\right\}}$ (cf. {\color{black}Proposition~}\ref{S4L40}). It then follows that for all $x>0$ and $u \in  {(s_0-\f{\eps\gamma}{2\pi}, s_0+\gamma+\f{\eps\gamma}{2\pi} )}$  fixed, the function $|\widetilde G(u+iv)x^{-u-iv}|$ is locally integrable with respect to $v$. Let us  see  now what is the asymptotic behaviour of $|\widetilde G(u+iv)|$ for $u \in  {(s_0-\f{\eps\gamma}{2\pi}, s_0+\gamma +\f{\eps\gamma}{2\pi} )}$ fixed  when  $|v|\to \infty$. This will be easier in terms of the variable $\zeta $  introduced in \eqref{changevarbles} and the function $F$ defined in  \eqref{CV}, whose expression was obtained in \eqref{def_F_prolong}:
{\color{black}$$F(\zeta)=\exp \Big(- \dfrac{1}{2i\pi}\dst\int_0^{+\infty} \log(\tilde\varphi(w))\Big\{ \dfrac{1}{w-\tilde \zeta}  -\dfrac{1}{w+\comM{1}}\Big\}\; dw \Big)\exp\left(\f{k}{2} \log \left(\varphi(\vert\zeta\vert)\right)\right),$$
$$ \tilde \zeta \in D(0),\; \zeta=\tilde \zeta e^{2ik\pi},$$}
where $\tilde\varphi$ is defined in  \eqref{S4Etildephi}. 

Notice that if $s=u+iv{\color{black}+k\gamma}$ with $u\in (s_0, s_0+\gamma )$ fixed, {\color{black} $k\in\Z$} and $v\in \R$, this yields in terms of the variable $\zeta =T(s)=e^{2i\pi \frac {s-s_0} {\gamma }}$:
\begin{eqnarray*}
\zeta &=&e^{-2\pi \frac {v} {\gamma }}e^{2i\pi \frac {u-s_0} {\gamma }}{\color{black}e^{2ik \pi}}=re^{i\theta}{\color{black}e^{2ik\pi}},\\
r&=&e^{-2\pi \frac {v} {\gamma }}\in (0, \infty),\\
\theta&=&2\pi \frac {u-s_0} {\gamma } \in (0, 2\pi ),\,\,\,\hbox{fixed}.
\end{eqnarray*}
Our goal is then to estimate  the behaviour of $F(\zeta )$ as $|\zeta |\to \infty$ and $|\zeta |\to 0$ and $\theta\in (0, 2\pi ),$ {\color{black} $k\in \Z$} fixed. 
By Lemma \ref{S4Estlog}
\begin{equation*}
F(\zeta )=\exp\left(-\frac {1} {2i\pi }\left(I({\color{black}\tilde \zeta} ){+\mathcal O(1) }\right) \right){\color{black}\exp\left(\f{k}{2} \log \left(\varphi(\vert\zeta\vert)\right)\right)},
\end{equation*}
where
\begin{eqnarray*}
I({\color{black}\tilde \zeta})& = &\int _0^\infty \log \left|\log w  \right|\left(\frac {1} {w-{\color{black}\tilde\zeta} }-\frac {1} {w+\comM{1}} \right)dw. 
\end{eqnarray*}

We have then, as $|\zeta |\to \infty$ or $|\zeta |\to 0$:
\begin{eqnarray*}
| F(\zeta)| &= \Big|\exp\left( - \dfrac{1}{2 i \pi  } I({\color{black}\tilde \zeta}){\color{black}+\f{k}{2}\log\left(\varphi(\vert\zeta\vert)\right)}{ +\mathcal O(1)} \Big\} \right) \Big|  \\ \\
&= \exp\left( -\dfrac{1}{2\pi } \Im m (I(\zeta)) {\color{black} +\f{k}{2} \Re e\left(\log (\varphi(\vert\zeta\vert))\right)}{+\mathcal O(1)}\right).
\end{eqnarray*}
{\color{black}By Lemma~\ref{S4Estlog}
\begin{equation*}
\log (\widetilde \varphi (\zeta ))=\log \left|\log {\vert} \zeta {\vert}  \right|+\mathcal O(1),\,\, \text{ for } \,\,  |\zeta | \to 0^+\text{ or } |\zeta |\to \infty.
\end{equation*}}
{\color{black}We can use this expression for $\zeta\to \vert \zeta\vert$ and obtain
\begin{equation*}
\Re e\big(\log (\varphi (\vert \zeta\vert ))\big)=\log \left|\log {\vert} \zeta {\vert}  \right|+\mathcal O(1),\,\, \text{ for } \,\,  |\zeta | \to 0^+\text{ or } |\zeta |\to \infty.
\end{equation*}}

By { Lemma \ref{Lemma_estimate_bounded_term} and Lemma \ref{Lemma_estimate_log}} 
$$
- \Im m (I(\zeta))=-\log \left|\log (|\zeta|)\right|\comM{(\pi-\theta)}
+\mathcal O(1), \quad \text{as $|\zeta |\to \infty$ or $|\zeta |\to 0$}.
$$
{\color{black} Finally we obtain
$$| F(\zeta)|=\exp\left( -\log |\log {\vert} \zeta {\vert} \f{\pi-\theta-2k\pi}{2\pi} +\mathcal O(1)\right) $$}

We deduce that for {\color{black}$\zeta$ such that $\theta \in (0,\pi)$ and $k\leq 0$, }
\begin{equation}
\label{S4E450}
| F(\zeta)| =o(1),\,\,\hbox{as}\,\,|\zeta |\to 0 \,\,\hbox{and}\,\,|\zeta |\to \infty.
\end{equation}
Using the change of variables \eqref{changevarbles}, \eqref{CV}, it follows from \eqref{S4E450}  that for all $u\in (\comM{s_0-\frac {\varepsilon \gamma } {2\pi }}, s_0+\comM{\dfrac{\gamma}{2} })$ fixed:
\begin{equation}
\label{S4E454}
| \widetilde G(s)| =o(1),\,\,\hbox{as}\,|\Im m (s)|\to  \infty,\,\,\Re e (s)=u.
\end{equation}

The function  $\widetilde G$ is analytic in the strip $\Re  e(s)\in (s_0-\frac {\varepsilon \gamma } {2\pi } , s_0+\frac {(2\pi +\varepsilon)\gamma  } {2\pi })$  and bounded as
$|s|\to \infty$ \comM{for $\Re e(s) \in ( s_0-\frac {\varepsilon \gamma } {2\pi } , s_0 + \dfrac{\gamma}{2} )$}.
Its inverse  Mellin transform  $\widetilde g$  is then uniquely defined as a distribution on $(0, \infty)$ by \eqref{gtilde} {in the sense of distributions}, where $u$ may take any value in the interval 
\comM{$ (s_0-\frac {\varepsilon \gamma } {2\pi }, s_0 + \dfrac{\gamma}{2} )$
using Theorem 11.10.1 in \cite{misra1986transform}.}

{\color{black} Let us now study the regularity of $\tilde g$.} \comM{
Let us assume $u<s_0$, {\color{black} {\it i.e.} for instance $k=-1$}. Then, there is $b > 1$ such that}
\begin{equation*}
|F(\zeta )|\le Ce^{-\frac {\comM{b\pi}} {2\pi } \log(\log|\zeta |)},\,\,\,|\zeta |\to 0 ,\,\,|\zeta |\to \infty,
\end{equation*}
Therefore,
\begin{equation*}
|\widetilde G(u+iv )|\le C|v|^{-\frac {\comM{b}} {2} },\,\,\,|v |\to \infty, \quad b>1,
\end{equation*}
and 
\begin{equation*}
\int  _{-\infty}^{\infty}\left | \widetilde G(u+iv)\right|^2dv<\infty.
\end{equation*}
This shows that for any {\color{black} $u<s_0$}  the function $\widetilde G_u(v): v\to  \widetilde G(u+iv)$ is such that $\widetilde G_u\in L^2(\R)$. It follows that its Fourier transform also belongs to $L^2(\R)$:
\begin{equation}
\label{S4E700}
\int  _{-\infty}^{\infty}\left | \int  _{-\infty}^{\infty}\widetilde G(u+iv)e^{-ivz}dv\right|^2dz<\infty.
\end{equation}
Using the change of variables $z=\log x$ we deduce:
\begin{eqnarray}
\int  _{-\infty}^{\infty}\left | \int  _{ -\infty}^{\infty}\widetilde G(u+iv)e^{-ivz}dv\right|^2dz=\int  _{-\infty}^{\infty}\left | \int  _{-\infty}^\infty\widetilde G(u+iv)e^{-iv\log x}dv\right|^2\frac {dx} {x} \nonumber \\
=\int  _{-\infty}^{\infty}\left | \int  _{ -\infty}^{\infty}\widetilde G(u+iv)x^{-\frac {1} {2}-iv}dv\right|^2dx.\label{S4E701}
\end{eqnarray}
Then, since
\begin{equation*}
\widetilde g(x)=\frac {1} {2\pi }\int  _{ v=-\infty}^{v=\infty}\widetilde G(u+iv)x^{-(u+iv)}dv,\,\,\,\,\,\, u\in (s_0, s_0+\gamma ),
\end{equation*}
it follows from  \eqref{S4E700} and \eqref{S4E701} that:
$$
\tilde g(x)=ix^{-u}\int  _{ -\infty}^{\infty}\widetilde G(u+iv)x^{-iv}dv,
$$
\begin{equation*}
|\tilde g(x)|= x^{-u+\frac {1} {2}}\left|\int  _{ -\infty}^{\infty}\widetilde G(u+iv)x^{-\frac {1} {2}-iv}dv\right|.
\end{equation*}
{\color{black} Hence $\tilde g (x) x^{u-\f{1}{2}} \in L^2_x$ as soon as $u<s_0$.}

\comM{
We recall that the value of $\tilde{g}$ defined by \eqref{gtilde} does not depend on $u$. {\color{black} Let us}
 choose $s_0>2$ and $\eps >0 $ sucht that $\tilde{g} \in L^1(x+ x^{\g+1})$.}

\comM{ {\bf Asymptotic behavior of $\tilde{g}$ around $x = +\infty$}.}
\begin{eqnarray*}
\int _1^\infty x^{1+ \gamma }|\tilde g(x)|dx= \int _1^\infty x^{1+ \gamma -u+\frac {1} {2}}\left|\int  _{ -\infty}^{\infty}\widetilde G(u+iv)x^{-\frac {1} {2}-iv}dv\right|dx\\
\le
\left( \int _1^\infty \left|\int  _{ -\infty}^{\infty}\widetilde G(u+iv)x^{-\frac {1} {2}-iv}dv\right|^2dx\right)^{1/2}
\left(\int _1^\infty x^{2(1+ \gamma  -u)+1}dx\right)^{1/2},
\end{eqnarray*}
the last integral in the right hand side being convergent whenever
\begin{eqnarray*}
2(1+ \gamma -u)+1<-1,\Longleftrightarrow \g+1 <u-1,
\end{eqnarray*}
\comM{i.e. whenever} we choose \comM{$s_0$ such that}
$$
2+\gamma <u<s_0.
$$
\comM{Notice that this also implies}  that
\begin{eqnarray*}
\int _1^\infty x|\tilde g(x)|dx<\infty.
\end{eqnarray*}

\comM{ {\bf Asymptotic behavior of $\tilde{g}$ around $x = 0$}.}

\begin{equation}
\label{ordre1}
\begin{aligned}
\int _0^1 x^{ }|\tilde g(x)|dx&= \int _0^1 x^{1  -u+\frac {1} {2}}\left|\int  _{ -\infty}^{\infty}\widetilde G(u+iv)x^{-\frac {1} {2}-iv}dv\right|dx\\
&\le
\left( \int _0^1 \left|\int  _{ -\infty}^{\infty}\widetilde G(u+iv)x^{-\frac {1} {2}-iv}dv\right|^2dx\right)^{1/2}
\left(\int _0^1 x^{2(1  -u)+1}dx\right)^{1/2}
\end{aligned}
\end{equation}
The last integral in the right hand side is convergent whenever
\begin{eqnarray*}
2(1  -u)+1>-1\Longleftrightarrow u<2, 
\end{eqnarray*}
then we need to choose $u<2$.
\comM{To be allowed to take $u<2$ with $2 + \gamma <s_0$, we need to impose 
that $\eps$ satisfies \begin{equation*}
 2 >u> s_0- \dfrac{\eps \gamma}{2\pi}> 2 + \gamma- \dfrac{\eps \gamma}{2\pi}, 
\qquad \text{ i.e. }\qquad  \eps >2 \pi. \end{equation*}
}
\comM{Notice that \eqref{ordre1} also implies}  that
$$
\int _0^1x^{1+\gamma }\tilde g(x)dx<\infty.
$$
Thus, \comM{if $s_0 >2+ \g$ and $\eps>2\pi$, the function $\tilde{g}$ defined by \eqref{gtilde} with $u \in (s_0 - \dfrac{\eps \gamma}{2\pi} , s_0 + \dfrac{\gamma}{2})$ satisfies $\tilde g\in L^1((x+x^{\gamma +1})dx)$. 
}
\end{proof} 

In order to prove that $\tilde g$ and $g$ are the same we first show the following lemma.
\begin{lemma}
\label{S4E567}  
The function $\tilde g$ defined by \eqref{gtilde} satisfies the equation \eqref{eq:g}.
\end{lemma}
\begin{proof}
By construction $\widetilde G$  is analytic in the region $\left\{ s\in \C; \Re e (s)\in (s_0-\frac {\varepsilon \gamma } {2\pi } , s_0+\frac {(2\pi +\varepsilon)\gamma  } {2\pi }) \right\}$ and satisfies
\begin{equation}
\label{S4EW1}
(2-s) \widetilde G(s) =  \a \g (K_0(s)-1) \widetilde G(s+\g)
\end{equation}
on that region. We take the inverse Mellin transform in both sides of \eqref{S4EW1} term by term. Since $(2-s)\tilde G(s)$ is again analytic in the strip $u\in (s_0-\frac {\varepsilon \gamma } {2\pi } , s_0+\frac {(2\pi +\varepsilon)\gamma  } {2\pi }) $ and $\tilde G(s)$ is bounded as $|s|\to \infty$ in that strip, we deduce that the inverse Mellin transform of $(2-s) \widetilde G(s)$ is a well-defined distribution. Moreover:
\begin{eqnarray*}
\frac {1} {2i\pi }\int _{\Re e(s)=u}(2-s) \widetilde G(s)x^{-s}ds=2\tilde g(x) -\frac {1} {2i\pi }\int _{\Re e(s)=u}s \widetilde G(s)x^{-s}ds\end{eqnarray*}
Using the well-known identity:
$$
\frac {1} {2i\pi }\int _{\Re e(s)=u}s \widetilde G(s)x^{-s-1}ds=-\frac {\partial } {\partial x}
\left( \frac {1} {2i\pi }\int _{\Re e(s)=u} \widetilde G(s)x^{-s}ds\right)
$$
we obtain:
\begin{equation*}
\frac {1} {2i\pi }\int _{\Re e(s)=u}(2-s) \widetilde G(s)x^{-s}ds=2\tilde g(x)+x\frac {\partial \tilde g} {\partial x}.
\end{equation*}
We consider now the Mellin transform of the right hand side of \eqref{S4EW1}. Since the function $K_0$ is analytic and bounded for $\Re e(s)>1$ and the function $\widetilde G(s)$ is analytic  for $\in \C$  and bounded for $\Re e(s)\in (s_0-\frac {\varepsilon \gamma } {2\pi } , s_0+\frac {(2\pi +\varepsilon)\gamma  } {2\pi })$, the inverse Mellin transform of 
$(K_0(s)-1)\widetilde G(s+\gamma )$ is a well-defined distribution for  $u>1$ and $u+\gamma \in (s_0-\frac {\varepsilon \gamma } {2\pi } , s_0+\frac {(2\pi +\varepsilon)\gamma  } {2\pi })$. Let us then choose $u\in(s_0-\frac {\varepsilon \gamma } {2\pi } , s_0+\frac {\varepsilon \gamma } {2\pi })$. Notice that $1<2<s_0-\frac {\varepsilon \gamma } {2\pi }$ from where $u>1$ too. Then,
\begin{eqnarray*}
 \frac {1} {2i\pi }\int  _{ \Re e(s)=u }(K_0(s)-1)\widetilde G(s+\gamma )x^{-s}ds=
 \frac {1} {2i\pi }\int  _{ \Re e(s)=u }K_0(s)\widetilde G(s+\gamma )x^{-s}ds- \\
 -\frac {1} {2i\pi }\int  _{ \Re e(s)=u }\widetilde G(s+\gamma )x^{-s}ds
\end{eqnarray*}
We first have:
\begin{eqnarray*}
\frac {1} {2i\pi }\int  _{ \Re e(s)=u }\widetilde G(s+\gamma )x^{-s}ds=\frac {1} {2i\pi }\int  _{ \Re e(\sigma )=u+\gamma  }\widetilde G(\sigma  )x^{-\sigma+\gamma   }ds\\
=\frac {x^\gamma } {2i\pi }\int  _{ \Re e(\sigma )=u+\gamma  }\widetilde G(\sigma  )x^{-\sigma}ds.
\end{eqnarray*}
Since $u\in(s_0-\frac {\varepsilon \gamma } {2\pi }, s_0+\frac {\varepsilon \gamma } {2\pi })$, we have
$u+\gamma \in (s_0+\frac {(2\pi -\varepsilon) \gamma } {2\pi } , s_0+\frac {(2\pi +\varepsilon)\gamma  } {2\pi })$ and then 
$$
\frac {x^\gamma } {2i\pi }\int  _{ \Re e(\sigma )=u+\gamma  }\widetilde G(\sigma  )x^{-\sigma}ds=\tilde g(x)
$$
from where
\begin{eqnarray*}
\frac {1} {2i\pi }\int  _{ \Re e(s)=u }\widetilde G(s+\gamma )x^{-s}ds=x^\gamma \tilde g(x).
\end{eqnarray*}

We use now the definition of $K_0$ in terms of $k_0$ to write:
\begin{eqnarray*}
\frac {1} {2i\pi }\int  _{ \Re e(s)=u }K_0(s)\widetilde G(s+\gamma )x^{-s}ds=
\frac {1} {2i\pi }\int  _{ \Re e(s)=u }\int _0^\infty k_0(y)y^{s-1}dy \widetilde G(s+\gamma )x^{-s}ds\\
=\int _0^\infty k_0(y) \left( \frac {1} {2i\pi }\int  _{ \Re e(s)=u } \widetilde G(s+\gamma )\left(\frac {x} {y}\right)^{-s}ds\right)\frac {dy} {y}.
\end{eqnarray*}
Using the same argument and the same choice of $u$ as before:
$$
\frac {1} {2i\pi }\int  _{ \Re e(s)=u } \widetilde G(s+\gamma )\left(\frac {x} {y}\right)^{-s}ds=\left(\frac {x} {y}\right)^{\gamma }\tilde g\left(\frac {x} {y}\right)
$$
from where:
\begin{eqnarray*}
\frac {1} {2i\pi }\int  _{ \Re e(s)=u }K_0(s)\widetilde G(s+\gamma )x^{-s}ds=
x^\gamma \int _0^\infty k_0(y) y^{-\gamma }\tilde g\left(\frac {x} {y}\right)\frac {dy} {y}.
\end{eqnarray*}
Using the change of variable $y=\frac {x} {z}$ we deduce
\begin{eqnarray*}
\frac {1} {2i\pi }\int  _{ \Re e(s)=u }K_0(s)\widetilde G(s+\gamma )x^{-s}ds=
\int _x^\infty k_0\left(\frac {x} {z}\right) z^{\gamma -1}\tilde g(z) dz.
\end{eqnarray*}
This shows  that the function $\tilde g$ satisfies the equation \eqref{eq:g} and proves Lemma \ref{S4E567}.
\end{proof}

We may now proceed to prove Theorem \ref{inv_mellin}.
\\

\begin{proof}[End of the proof of Theorem \ref{inv_mellin}] By Lemma  \ref{S4E567}  the function $g$ satisfies the equation   \eqref{eq:g} and by Lemma \ref{S4L1},  $g\in L^1((x+x^{\gamma +1}))$. Then by the uniqueness Theorem 3.1 of  \cite{EMR05}, p. 110 we deduce that $\tilde g=g$. 
\end{proof}

\begin{proof} [\textbf{Of  Theorem \ref{Well_posedness_inverse_pb}}]
By Proposition  \ref{S4L40} the function  $\widetilde G$ is analytic and bounded on the domain 
 $D=\left\{s\in \C;  \Re  e(s)\in (s_0-\frac {\varepsilon \gamma } {2\pi } , s_0+\frac {(2\pi +\varepsilon)\gamma  } {2\pi })\right\}$ and then, by classical properties of the Mellin transform (cf. Theorem 11.10.1 in \cite{misra1986transform}) $\widetilde G(s)=M[\tilde g](s)$, for $s\in D$. Since by Theorem \ref{inv_mellin}  $g=\tilde g$, we deduce that $G=\widetilde G$ on $D$. It follows in particular that $|G|$ does not vanish on  $D$, and this proves point (i). 
 
We may then divide both terms of equation \eqref{Formula} by $G(s+\gamma )$ to obtain equation \eqref{S2E101}  and this shows the point (ii). 

Since the function {$K_0$ is bounded,  applying Theorem 11.10.1 in~\cite{misra1986transform} we have
$$
k_0(x)=\frac {x^2} {2i\pi } \f{d^2}{dx^2} \int  _{ \Re e(s)=s_0 }\f{x^{-s}}{s^2}K_0(s)ds,
$$
that we write also, {in the sense of distributions}
$$
k_0(x)=\frac {1} {2i\pi }  \int  _{ \Re e(s)=s_0 }{x^{-s}}K_0(s)ds,
$$
and applying point (ii) we get point (iii)}.
\end{proof}


\section{Conclusion}

{In this study, we provided  a first theoretical ground to the question of estimating the function parameters of a pure fragmentation equation from its solution. 
To this purpose, we departed from its self-similar asymptotic profile, along the lines of previous studies carried out for the growth-fragmentation equation~\cite{BDE14,PZ}. }

{We proved two main results: uniqueness for the fragmentation rate and kernel, and a reconstruction formula for the fragmentation kernel 
based on the Mellin transform of the equation. }
{The most delicate point lies in the proof of the reconstruction formula. 
This requires to prove that the Mellin transform of the asymptotic profile does not vanish on a vertical strip of the complex plane - a property far from obvious achieved with the use of the {Cauchy integral} and a careful study of the asymptotic behaviour of the function on vertical lines  of the complex plane.}

{With these results however, the inverse problem of reconstructing the function parameters of the fragmentation equation is far from being solved in practice. 

First, stability of the reconstruction formula (iii) of Theorem~\eqref{Well_posedness_inverse_pb} needs to be studied in an adapted space, and this inverse problem appears as severely ill-posed, as most problems of deconvolution type. { Stability could then lead to error estimates, to take into account the fact that  the asymptotic profile is measured with a certain noise in a certain space.

To go further, it would also be of interest to take into account a convenient  statistical modelling of the measurement noise. A natural one, mimicking the experiments carried out in~\cite{XR13}, would be to assume that at times $t_i,$ the measurement consists in a sample of fibrils, whose sizes $(x_1^i,\cdots x_{n_i}^i)$ are measured. A first assumption would then be, in the spirit of~\cite{DHRR}, to assume that these samples are realizations of i.i.d. random variables, whose density $f(t_i,x)$ satisfies the fragmentation equation. This leads to other difficult and interesting questions in statistics: indeed, the samples are naturally {\it not} independent, but a theoretical justification of this assumption could be investigated, as done for instance in~\cite{Adelaide} for the case of an age-structured process. }


Concerning the fragmentation rate, we shall need a new estimation method, since ours strongly uses the behaviour of the asymptotic profiles for very large sizes, what is out of reach in practice. 
Finally, numerical tests and application to real data shall be carried out in a future work.
} 
%
%

%

\section{Appendices}

\subsection{Behaviour of $K_0(s)$ for   $s\in (0, \infty)$, large.}
\label{app:K0}

The following lemma states that if $k_0$ is continuous in a neighbourhood of $x=1$, the Mellin transform $K_0$
 converges to $0$ as $1/s$ when $\Re e(s)$ goes to $\infty$.
 The proof is a variation on the Laplace method \cite{Rouviere} through the change of variable 
 $ z=e^{-x}.$
\begin{lemma}
\label{meth_laplace}
 Under hypothesis~\eqref{hyp4}:
\begin{equation}
\label{lem:K0}
 K_0(s) = \dfrac{k_0(1)}{s}+ o\left(\dfrac{1}{s}\right), \qquad s\to \infty, \,\, s\in \mathbb{R}.
\end{equation}
If we also assume that, for some $\varepsilon >0$ there exists $r>1$ and $C_r>0$ such that, for all $x\in (1-\varepsilon , 1)$:
$$
|k_0(x)-k_0(1)|\le C_r|\log (x)|^r
$$
then
\begin{equation}
\label{lem:K0BIS}
 K_0(s) = \dfrac{k_0(1)}{s}+ o\left(\dfrac{1}{s^r}\right), \qquad s\to \infty, \,\, s\in \mathbb{R}.
\end{equation}
\end{lemma}

\begin{proof}
 Let us evaluate 
 the limit as $s \rightarrow \infty $ of the following expression
 
 \begin{equation}
 \label{S6E129}
  s K_0(s) = s\dst\int_0^{1-\eps}k_0(x)x^{s-1}\;dx + s\dst\int_{1-\eps}^{1}k_0(x)x^{s-1}\;dx .
 \end{equation}
We first notice that:
\begin{equation}
\begin{aligned}
\left| s\dst\int_0^{1-\eps}k_0(x)x^{s-1}\;dx \right| &\leq s (1-\eps)^{s-1} \dst\int_0^{1-\eps}k_0(x) dx \\
&\leq 2 s \exp\left((s-1) \ln(1-\eps) \right) \underset{s \rightarrow \infty }{\longrightarrow} 0,
\end{aligned}
\end{equation}
and  so the first term in the right hand side of \eqref{S6E129} goes to zero exponentially fast.  On the other hand,
using the change of variable $x=y^{1/s}$:
\begin{equation}
\begin{aligned}
s\dst\int_{1-\eps}^{1}k_0(x)x^{s-1}\;dx = \dst\int_{(1-\eps)^{\textcolor{black}{s}}}^{1} k_0(y^{1/s}) dy
\end{aligned}
\end{equation}
Since by  \eqref{hyp4} $k_0$ is continuous on $[1-\varepsilon , 1] $, it follows that, for all $y\in (1-\varepsilon , 1)^{s}$:
\begin{eqnarray*}
&&(i)\qquad k_0(y^{1/s})\le \max _{ x\in [1-\varepsilon , 1] } k_0(x)<\infty\\
&&(ii)\qquad  \lim _{ s\to \infty  } k_0(y^{\frac {1} {s}})=k_0(1)\\
&&(iii)\qquad \lim  _{ s\to \infty  } (1-\varepsilon )^s=0
\end{eqnarray*}
and we deduce by the Lebesgue's convergence Theorem:
$$
\lim _{ s\to \infty  }s\dst\int_{1-\eps}^{1}k_0(x)x^{s-1}\;dx= k_0(1).
$$
This shows \eqref{lem:K0}. In order to prove \eqref{lem:K0BIS} we  use that
$$
s\dst\int_{1-\eps}^{1}x^{s-1}\;dx=1-(1-\varepsilon )^s
$$
to write:
\begin{eqnarray*}
\left|s\dst\int_{1-\eps}^{1}k_0(x)x^{s-1}\;dx-k_0(1)\right|\le s\dst\int_{1-\eps}^{1}|k_0(x)-k_0(1)|x^{s-1}\;dx+k_0(1)(1-\varepsilon )^s\\
\le C_rs\dst\int_{1-\eps}^{1}(\log x) ^rx^{s-1}\;dx+k_0(1)(1-\varepsilon )^s.
\end{eqnarray*}
It may be checked that
$$
\int_{1-\eps}^{1}(\log x) ^rx^{s-1}\;dx=s^{-r}\left( \Gamma \left(1+r, -s\log(1-\varepsilon ) \right)-\Gamma (1+r)\right)
$$
and since $s$  $\left( \Gamma \left(1+r, -s\log(1-\varepsilon ) \right)-\Gamma (1+r)\right)$ is bounded as a function of $s$, \eqref{lem:K0BIS} follows.
\end{proof}

\begin{remark}
 If $k_0$  { satisfies \eqref{hyp2}, \eqref{hyp4}, \eqref{hyp3},  and if moreover $k_0$} $\in L^1(0,1)$, the Riemann Lebesgue theorem guarantees us that
 \begin{equation}\label{raj}
K_0(s) \underset{Im(s) \rightarrow {\pm} \infty}{\longrightarrow} 0.
 \end{equation}
If $k_0$ is a general measure, this property may not be true anymore. The measures such that
\eqref{raj} is satisfied are known as the Rajchman measures.
See for example \cite{Rajchman} for a characterization of the Rajchman measures in term 
of what sets they annihilate ({\it i.e.} give measure zero).
\end{remark}

\subsection{The Sokhotsky-Plemelj formula}\label{sec:Sokh}
The Sokhotsky-Plemelj formula {(see a proof in \cite{misra1986transform} page 33)}
is an  identity among distributions which states \comM{in one of its variants} that 
\begin{lemma}[Sokhotsky-Plemelj formula]
For $a\in \R$, 
\begin{equation}
 \lim\limits_{\eps \rightarrow 0 }\;  \dfrac{1}{w-a {e^{\pm i\eps}}}  = P.V.\; \dfrac{1}{w-a}\;  \pm\; i \pi \delta(w-a), \quad \mathcal{D}'(\mathbb{R}), 
\end{equation}
or in other terms, for {a test function $f$} 
\begin{equation}
\lim\limits_{\eps \rightarrow 0 } \dst\int_{0}^{\infty} \dfrac{f(w)}{w-a {e^{\pm i\eps}}} dw = P.V. \dst\int_0^{\infty}\dfrac{f(w)}{w-a} dw 
 \pm\; i \pi f(a), \qquad a \in \mathbb{R}.
\end{equation}
\end{lemma}

\subsection{{Three} auxiliary Lemmas.}\label{sec:append:3lemmas}
The first lemma gives an estimate on $\log(\tilde\varphi(w))$ as $|\zeta |$ goes to $0$ or $+\infty$ with a fixed argument $\theta$.
\begin{lemma}
\label{S4Estlog}[Estimate on $\log(\varphi)$]
Suppose that  
the  determination of the logarithm function is chosen   as in \eqref{S4Elog}.
Then \comM{for $s_0$ large enough}, {and $\zeta \in \C\setminus\R^+$ }:
\begin{equation}
\log (\widetilde \varphi (\zeta ))=\log \left|\log {\vert} \zeta {\vert}  \right|+\mathcal O(1),\,\, \text{ for } \,\,  |\zeta | \to 0^+\text{ or } |\zeta |\to \infty.
\end{equation}
\label{lemma_estimate_log_phi}
\end{lemma}

\begin{proof}
By definition, for any $\zeta \in \C {\setminus \R^+}$ {(identified with the first sheet of the Riemann's surface)}, there is a unique $s\in\C$ with $\Re e(s)\in (s_0, s_0+\gamma )$ such that   $\tilde \varphi (\zeta )=\Phi (s)$. 
We have
\begin{equation}
 \Phi (s) = \frac {s-2} {\alpha \gamma (1-\color{black}{K_0(s)})}.
\end{equation}
\color{black}
{Since $K_0(2)=1$ and $k_0$ is supported over $[0,1]$ and satisfies Hyp-4, we have,
}
for any $s$ such that 
$\Re e(s)>2$
\begin{equation*}
 |K_0(s)| \leq K_0(\Re e(s)) <1.
\end{equation*}
Then, for $\Re e(s)>2$ fixed, there is a contant $C$ such that
\begin{equation*}
  \Phi (s) = C s +  O(1),\quad \text{ as } \Im m(s) \to \pm \infty,
\end{equation*}
which is by the definition~\eqref{S4Etildephi} of~$\tilde \varphi$,
\begin{equation*}
 \widetilde{\varphi}(\zeta) = \Phi\left(s_0 +\dfrac{\g}{2i\pi}\log(\zeta) \right)= \dfrac{C\g}{2 i \pi} \log(\zeta) + O(1),\quad \text{ as }|\zeta| \to 0^+,\; |\zeta| \to  \infty.
\end{equation*}
 The expression of  $\log (\widetilde\varphi (\zeta))$ is then given by
\begin{equation}\label{bhv}
\log (\widetilde\varphi (\zeta))=\log \left|\log|\zeta | \right|+{O(1)}, \quad \hbox{as}\,\, |\zeta| \to 0^+\,\,\, \text{and}\,\,|\zeta |\to \infty.
\end{equation}
 \color{black}


This ends the proof of Lemma \ref{lemma_estimate_log_phi}.
\end{proof}

To describe the asymptotic behaviour of $F$ for $|\zeta| \rightarrow 0$ and $|\zeta| \rightarrow + \infty$, 
we need to understand the behaviour of the imaginary part of 
\begin{equation}
\color{black}{\int _0^\infty \log \tilde{\varphi}(w)\left(\frac {1} {w-\zeta }-\frac {1} {w+1} \right)dw
}
\end{equation}

as $|\zeta| \rightarrow 0$ and $|\zeta| \rightarrow + \infty$ {\color{black} {for $\zeta \in D(0)=\C \setminus \R^+$}}. 
\color{black}{This is done through the following two lemmas.

\begin{lemma}\label{Lemma_estimate_bounded_term}
For $\zeta \in D(0)$,
 denoting  $\zeta =re^{i \theta}, \; \theta \in (0,2\pi)$,
\begin{equation}
\Im m\left( \int _0^\infty \left(\frac {1} {w-\zeta }-\frac {1} {w+1} \right)dw\right) = O (1), \qquad |\zeta|\to 0, \; |\zeta|\to \infty.
\end{equation}
\end{lemma}
\begin{proof}
For $|\zeta|$ small, the Lebesgue dominated convergence theorem guarantees that
\begin{equation*}
\int_1^{+\infty}  \left(\frac {1} {w-\zeta }-\frac {1} {w+1} \right)dw \underset{|\zeta| \to 0}{\to}\int_1^{+\infty}  \left(\frac {1} {w }-\frac {1} {w+1} \right)dw =  O(1), \qquad |\zeta|\to 0. 
 \end{equation*}
Then, we write
\begin{equation*}
\int_0^{1}  \left(\frac {1} {w-\zeta }-\frac {1} {w+1} \right)dw = \int_0^1 \dfrac{dw}{w-\zeta}  -  \int_0^1 \dfrac{dw}{w+1} , 
\end{equation*}
and a straightforward integration gives us 
\begin{equation*}
 \int_0^{1} \frac {1} {w-\zeta} dw= \log(1-\zeta)- \log(-\zeta) = -\log(-\zeta) + O(1), \qquad |\zeta|\to 0,
\end{equation*}
Hence
\begin{equation*}
 \Im m \left( \int _0^\infty \left(\frac {1} {w-\zeta }-\frac {1} {w+1} \right)dw \right) = -\arg(-\zeta) + O(1) = O(1),  \qquad |\zeta|\to 0.
\end{equation*}


For $|\zeta|$ large, 
the Lebesgue dominated convergence theorem guarantees that
\begin{equation*}
\int_0^{1}  \left(\frac {1} {w-\zeta }-\frac {1} {w+1} \right)dw \underset{|\zeta| \to \infty}{\to} - \int_0^{1}  \frac {dw} {w+1} dw =  O(1), \qquad |\zeta|\to \infty. 
\end{equation*}
Then, we write
\begin{equation*}
 \int _1^\infty \left(\frac {1} {w-\zeta }-\frac {1} {w+1} \right)dw=
  \int_1^{2r} \left(\frac {1} {w-\zeta }-\frac {1} {w+1} \right)dw
  +   \int_{2r}^{+\infty} \left(\frac {1} {w-\zeta }-\frac {1} {w+1} \right)dw,
\end{equation*}
where we recall that $r= |\zeta|$. 
We deal with the first term using a straightforward integration
\begin{equation*}
 \int_1^{2r} \left(\frac {1} {w-\zeta }-\frac {1} {w+1} \right)dw = \log(2r- \zeta) - \log(1-\zeta) - \log(2r+1) + \log(2),
\end{equation*}
thus
\begin{equation*}
\Im m\left( \int_1^{2r} \left(\frac {1} {w-\zeta }-\frac {1} {w+1} \right)dw\right) = \arg(2 - e^{i\theta}) - \arg(1- r e^{i\theta}) = O(1) , \quad |\zeta|\to \infty.
\end{equation*}
The second term needs more details. We write
\begin{equation*}
\int_{2r}^{\infty} \left(\frac {1} {w-\zeta }-\frac {1} {w+1} \right)dw 
 = \int_{2r}^{\infty} \dfrac{\zeta}{w(w-\zeta)}dw 
  +  \int_{2r}^{\infty} \dfrac{dw}{w(w-\zeta)}
  -  \int_{2r}^{\infty} \dfrac{\zeta+1}{w(w-\zeta)(w+1)}dw.
\end{equation*}
The Lebesgue theorem gives us 
\begin{equation*}
   \int_{2r}^{\infty} \dfrac{dw}{w(w-\zeta)}
  -  \int_{2r}^{\infty} \dfrac{\zeta+1}{w(w-\zeta)(w+1)}dw \underset{|\zeta| \to \infty}{\to} 0,
\end{equation*}
and we have
\begin{equation*}
\int_{2r}^{\infty} \dfrac{\zeta}{w(w-\zeta)}dw =
\int  _{ 2r }^{\infty} \frac {\zeta } {w^2 \left(1-\frac {\zeta } {w }\right) }dw 
=
\sum _{ k=0 }^\infty  \zeta ^{k+1} \int  _{ 2r }^{\infty}\frac {dw} {w^{k+2} }
= \sum _{ k=1 }^\infty \dfrac{1}{k}\left(\dfrac{e^{i\theta}}{2}\right)^k
 = -\Log\left(1- \dfrac{e^{i\theta}}{2}\right), \end{equation*}
Using the serie expansion
  \begin{equation*}
 \sum_{ k=1 }^\infty \frac {z^k} {k}
=-\Log(1-z),
\end{equation*}
{\color{black} where $\Log(z)$ is the principal determination of the logarithm (taken with $\Arg (z) \in (-\pi,\pi]$)}.
Hence
\begin{equation*}
 \Im m \left( \int _0^\infty \left(\frac {1} {w-\zeta }-\frac {1} {w+1} \right)dw \right) = -\Arg\left(1- \dfrac{e^{i\theta}}{2}\right) + O(1) = O(1),  \qquad |\zeta|\to \infty.
\end{equation*}
This ends the proof of Lemma \ref{Lemma_estimate_bounded_term}.
\end{proof}
The next Lemma gives the asymptotic behaviour of the imaginary part of}
\color{black}
the  following integral $I(\zeta)$
\begin{eqnarray*}
I(\zeta)& = &\int _0^\infty \log \left|\log w  \right|\left(\frac {1} {w-\zeta }-\frac {1} {w+1} \right)dw.
\end{eqnarray*}

\begin{lemma}\label{Lemma_estimate_log}
For $\zeta \in D(0)$,
\comM{ denoting  $\zeta =re^{i \theta}, \; \theta \in (0,2\pi)$
\begin{equation}\label{Im1z1}
\Im m (I(\zeta ))=\log \left|\log (r)\right| \comM{(\pi-\theta)} +\mathcal O\left( 1\right),\,\,\,\text{as  $|\zeta |\to 0$ and $|\zeta |\to \infty$}.
\end{equation}
}
\end{lemma}

The proof of Lemma \ref{Lemma_estimate_log}    uses the  following expressions, for all  $A>0$, {where $\gamma_E$ denotes the Euler's constant}:
\begin{align}
\label{app1_1}
\int _0^{A}\log \left|\log w  \right| w^kdw
&=\frac {-\Psi ((k+1)\log(A))+(A)^{k+1}\log \left|\log (A)  \right|} {k+1}, \qquad k \ge 0,\\
\label{app1_2}
\int_{A}^{1}\log \left|\log w  \right| \frac {dw} {w^{k+1}}
&=\frac {\gamma_E -\Psi (-k\log(A))+\log k+(A)^{-k}\log|\log(A)|} {k}, \qquad k\not =0, \\
\label{app1_3}
\int_{A}^{1}\log \left|\log w  \right| \;\frac {dw} {w}&=-\log(A)(-1+\log|\log(A)|, \\
\label{app1_4}
\int _1^{A}\log \left|\log w  \right| w^kdw
&=\frac {\gamma_E +\log(k+1)-\Psi ((k+1)\log(A))+(A)^{k+1}\log \left|\log (A)  \right|} {k+1}, \qquad k \ge 0,\\
\label{app1_5}
\int_{ A}^{\infty} \log \left|\log w  \right|\frac {dw} {w^{k+2} }
&=\frac {-\Psi (-(k+1)\log (A))+(A)^{-(k+1)}\log|\log(A)|} {k+1}, \qquad k\ge 0.
\end{align}
where 
\begin{equation*}
\Psi (z)=-P.V. \int  _{ -z }^\infty \frac {e^{-t}dt} {t}, \qquad \Psi \left( \dfrac{1}{x} \right) \underset{x\to 0}{\sim} \exp\left( \frac{1}{x}\right) \left( x +O(x^2) \right).
\end{equation*}

\subsubsection*{Proof of Lemma \ref{Lemma_estimate_log}}

Lemma \ref{Lemma_estimate_log} is describing the asymptotic behaviour of the imaginary part of $I(\zeta)$ as 
$|\zeta |\to 0$ and $|\zeta |\to + \infty$, with $\arg(\zeta) \equiv \theta\in (0, 2\pi )$, {{and extends easily then to $\arg(\zeta)=0$}}.

  {\bf \boxed{Step\,\, I.} Limit as $|\zeta |\to 0$.}
We split the integral $I(\zeta)$ in three terms
\begin{equation}
\label{S6E500}
\begin{aligned}
I(\zeta )=I _{ 1 }(\zeta )+I _{ 2 }(\zeta )+I _{  3 }(\zeta ),\\
I_{1}(\zeta )=\int _0^1 \log \left|\log w  \right|\frac {dw} {w-\zeta }, 
\qquad &I_{2}=-\int _0^1 \log \left|\log w  \right|\frac {dw} {w+1}=\frac {(\log 2)^2} {2}, \\
 &I_{3}(\zeta )=\int _1^\infty \log \left|\log w  \right|\left(\frac {1} {w-\zeta }-\frac {1} {w+1} \right)dw.\\
\end{aligned}
\end{equation}
We first notice that, by Lebesgue's convergence theorem $I_{3}(\zeta)$ converges towards a finite real limit:
 \begin{eqnarray}
 \label{S6E501}
\lim _{ |\zeta |\to 0 }I _{  3 }(\zeta )= \int _1^\infty\frac {\log \left|\log w  \right|} {w(w+1)}dw,
\end{eqnarray}
so that the behaviour of $I$ is dominated by the behaviour of $I_{1}(\zeta)$. We  cut it into three pieces:
\begin{eqnarray*}
{I_1}(\zeta )
&=&\int _0^{r/2} \log \left|\log w  \right| \frac {dw} {w-\zeta }
+\int_{ r/2 }^{2r} \log \left|\log w  \right| \frac {dw} {w-\zeta }+\int  _{ 2r }^1 \log \left|\log w  \right| \frac {dw} {w-\zeta }\\ \\
&=&{I_{1,1}(\zeta )\quad+\quad I_{1,2}(\zeta )\quad+\quad I_{1,3}(\zeta )}.
\end{eqnarray*}

{\subsubsection*{Study of the first integral $I_{1,1}(\zeta )=\int _0^{r/2} \log \left|\log w  \right| \frac {dw} {w-\zeta }$}}

Since in the first integral $0<w<|\zeta |/2$, we may write
\begin{eqnarray*}
\int _0^{r/2} \log \left|\log w  \right| \frac {dw} {w-\zeta }=\int _0^{r/2} \log \left|\log w  \right| \frac {dw} {\zeta (\frac {w } {\zeta }-1)}
&=&-\sum_{ k=0 }^\infty \zeta ^{-(k+1)}\int _0^{r/2}\log \left|\log w  \right| w^kdw. 
\end{eqnarray*}
Let us recall that we defined $r:= |\zeta|$. Using Formula \eqref{app1_1} for $k \geq 0$ {\color{black} and $A=\f{r}{2}$}, we have
\begin{equation}
\begin{aligned}
\int _0^{r/2}\log \left|\log w  \right| w^kdw
&=\frac {r^{k+1}\log \left|\log (r/2)  \right|} {(k+1)2^{k+1}}-\\
&-
\left(\frac {r} {2}\right)^{k+1}\left(\frac {1} {(k+1)^2\log(r/2)}+\mathcal O\left(\frac {1} {(k+1)^3\log(r/2)^2} \right) \right),
\end{aligned}
\end{equation}
as $|\zeta |\to 0$,  then 
\begin{equation*}
\begin{aligned}
\int _0^{r/2} \log \left|\log w  \right|& \frac {dw} {w-\zeta }=-\log \left|\log (r/2)  \right|\sum_{ k=0 }^\infty \left(\frac {|\zeta |} {\zeta }\right)^{k+1}\frac {1} {(k+1)2^{k+1}}+\\
&+\frac {1} {\log(r/2)}\sum_{ k=0 }^\infty \left(\frac {|\zeta |} {\zeta }\right)^{k+1}
\frac {1} {2^{k+1}(k+1)^2}\left(1+\mathcal O\left(\frac {1} {(k+1)\log(r/2)} \right) \right),\,\,|\zeta |\to 0.
\end{aligned}
\end{equation*}
Using the serie expansion
  \begin{equation*}
 \sum_{ k=1 }^\infty \left(\frac {\zeta } {2 |\zeta |} \right)^k\frac {1} {k}
=-{\color{black}\Log}\left(1-\frac {\zeta } {2|\zeta |} \right)
\end{equation*}
{\color{black} where $\Log(z)$ is the principal determination of the logarithm (taken with $\Arg (z) \in (-\pi,\pi]$)}
\begin{equation}\label{C1}
\comM{I_{1,1}(\zeta)=} \int _0^{r/2} \log \left|\log w  \right| \frac {dw} {w-\zeta }=\log \left|\log (r/2)  \right| {\color{black}\Log}\left(1-\frac {|\zeta |} {2\zeta } \right)+
\mathcal O\left(\frac {1} {|\log r|} \right),\,\,|\zeta |\to 0.
\end{equation}
{\subsubsection*{Study of the third integral $I_{1,3}(\zeta )=\int _{2r}^{1} \log \left|\log w  \right| \frac {dw} {w-\zeta }$}}

Similarly, using that $2|\zeta |<w$ in the third integral, we write:
\begin{eqnarray*}
\int _{2r}^{1} \log \left|\log w  \right| \frac {dw} {w-\zeta }=\int _{2r}^{1} \log \left|\log w  \right| \frac {dw} {w (1-\frac {\zeta  } {w })}
&=&\sum_{ k=0 }^\infty \zeta ^{k}\int _{2r}^{1}\log \left|\log w  \right| \frac {dw} {w^{k+1}}, 
\end{eqnarray*}
from where, using formula \eqref{app1_2} for $k\not =0$ {\color{black} and $A=2r$}:
\begin{eqnarray*}
\int _{2r}^{1}\log \left|\log w  \right| \frac {dw} {w^{k+1}}&=&\frac {(2r)^{-k}\log|\log(2r)|} {k}+\frac {\gamma_E +\log k} {k}+\frac {(2r)^{-k}} {k^2\log(2r)}\left(1+\mathcal O\left(\frac {1} {k\log(2r)} \right) \right),
\end{eqnarray*}
as $|\zeta |\to 0$, and then
 \begin{eqnarray*}
 \int _{2r}^{1} \log \left|\log w  \right| \frac {dw} {w-\zeta }&=&-\log(2r)(-1+\log(-\log(2r))+\log|\log(2r)|\sum_{ k=1 }^\infty \left(\frac {\zeta } {2 |\zeta |} \right)^k\frac {1} {k}+\\
 &&+\sum_{ k=1 } ^\infty\zeta ^k\left\{\frac {\gamma_E +\log k} {k}
 +\frac {(2r)^{-k}} {k^2\log(2r)}\left(1+\mathcal O\left(\frac {1} {-k\log(2r)} \right) \right)\right\}, \,\,|\zeta |\to 0.
 \end{eqnarray*}

we obtain:
\begin{equation}\label{C2}
\begin{aligned}
\int _{2r}^{1} \log \left|\log w  \right| \frac {dw} {w-\zeta }=-\log(2r)(-1+\log(-\log(2r))
-\log|\log(2r)|{\color{black}\Log}\left(1-\frac {\zeta } {2|\zeta |} \right)-\\
-\gamma_E {\color{black}\Log}(1-\zeta ) +\sum_{ k=1 } ^\infty\zeta ^k\left\{\frac {\log k} {k}
+\frac {(2r)^{-k}} {k^2\log(2r)}\left(1+\mathcal O\left(\frac {1} {k\log(2r)} \right) \right)\right\} ,\,\,|\zeta |\to 0.
\end{aligned}
\end{equation}
\comM{
which implies
\begin{equation}\label{C2}
\begin{aligned}
I_{1,3}(\zeta)=\int _{2r}^{1} \log \left|\log w  \right| \frac {dw} {w-\zeta }&=-\log(2r)(-1+\log(-\log(2r))
\\& -\log|\log(2r)|{\color{black}\Log}\left(1-\frac {\zeta } {2|\zeta |} \right)
+ \mathcal O\left(\frac {1} {|\log(r)|} \right).
\end{aligned}
\end{equation}
}
{\subsubsection*{Study of the second integral $I_{1,2}(\zeta )=\int _{r/2}^{2r} \log \left|\log w  \right| \frac {dw} {w-\zeta }$}}

The argument for the second integral is slightly different. We first  make the change of coordinates $w=rx$, $r=|\zeta |$, and obtain
\begin{eqnarray}
\int _{r/2}^{2r} \log \left|\log w  \right|\frac {dw} {w-\zeta }&=&\int _{1/2}^{2} \log \left|\log r x  \right|\frac {dx} {x-\frac {\zeta } {|\zeta |}}=\int _{1/2}^{2} \log \left|\log r +\log x  \right|\frac {dx} {x-\frac {\zeta } {|\zeta |}}\nonumber\\
  &=& \int _{1/2}^{2} \log \left|\log r\left(1+\frac{\log x}{\log r}\right)  \right|\frac {dx} {x-\frac {\zeta } {|\zeta |}}\nonumber\\
 &=& \int _{1/2}^{2}\log \left|\log r\right| \frac {dx} {x-\frac {\zeta } {|\zeta |}}
  +\int _{1/2}^{2} \log \left|\left(1+\frac{\log x}{\log r}\right)  \right|\frac {dx} {x-\frac {\zeta } {|\zeta |}}\nonumber\\
&=&\log \left|\log r \right| \left(\Log\left(2-\frac {\zeta } {|\zeta |}\right)-\Log\left(\frac {1} {2}-\frac {\zeta } {|\zeta |}\right)\right)+\nonumber \\
&&\hskip 4cm +\int _{1/2}^{2} \log \left|\left(1+\frac{\log x}{\log r}\right)  \right|\frac {dx} {x-\frac {\zeta } {|\zeta |}}.\label{S6E245}
\end{eqnarray}
Since we want to consider values of the argument of $\theta$ in the interval {\color{black}$(0, 2\pi) $},  the denominator  $x-\frac {\zeta } {|\zeta |}$ may then be {\color{black} close to} zero for $x=1$.  
Suppose then that $\arg (\zeta )=\theta$.
We will consider separately the case where $\cos\theta$ is close to one and the case where $\cos\theta$ is bounded away from one. Let us consider first the case where  $\cos\theta \ge 2/3$. We use  the change of variables $y=x-e^{i\theta}$ and obtain the expression
\begin{eqnarray}
\label{S6E123}
\int _{1/2}^{2} \log \left|\left(1+\frac{\log x}{\log r}\right)  \right|\frac {dx} {x-e^{i\theta}}=
\int _{\Gamma _\theta}\log \left|\left(1+\frac{\log (y+e^{i\theta})}{\log r}\right)  \right|\frac {dy} {y}
\end{eqnarray}
where
$$
\Gamma _\theta=\left\{y\in \C; \Re e (y)\in \left(\frac {1} {2}-\cos \theta, 2-\cos \theta\right),\,\,\Im m (y)=-\sin \theta \right\}.
$$
This integral may be written as follows
\begin{eqnarray}
\label{S6EVP}
\int _{\Gamma _\theta}\log \left|\left(1+\frac{\log (y+e^{i\theta})}{\log r}\right)  \right|\frac {dy} {y}=
\lim _{ \delta  \to 0 }\int _{\Delta _\theta(\delta  )}\log \left|\left(1+\frac{\log (y+e^{i\theta})}{\log r}\right)  \right|\frac {dy} {y}
\end{eqnarray}
where
$$
\Delta _\theta(\delta  )=\Gamma _\theta \setminus \left\{y\in  \Gamma _\theta; \,\, \Re e (y)\in (-\delta  , \delta  )\right\}.
$$
Define finally the sets
\begin{eqnarray*}
\Sigma _\theta (\delta  )=\left\{ y\in \Delta _\theta(\delta  ); \,\, \Re e (y)\in ({\color{black} 1/2-\cos\theta}, {\color{black}-1/2+\cos\theta })\right\}\\
\mathcal Q =\left\{y\in \C; \comM{\Re e (y)\in  (1/2-\cos\theta, -1/2+\cos\theta) },\,\Im m (y)\in (-1, 1)  \right\}.
\end{eqnarray*}
Notice that $\Sigma _\theta (\delta  )\subset \mathcal Q$. 
Due to the symmetry of $\Sigma _\theta (\delta  )$ with respect to the line $\Re e (y)=0$ we have:
\begin{equation}
\begin{aligned}
\int _{ \Sigma _\theta (\delta  ) }\frac {dy} {y}&= \Log\left(\dfrac{1}{2} - \cos\theta - i \sin\theta \right)-\Log\left(-\dfrac{1}{2} + \cos\theta - i \sin\theta \right) + \\
&\hspace{3cm}+ \Log\left(-\delta-i\sin\theta\right)-\Log\left(\delta-i\sin\theta\right)  \\
&= i\left(2 \Arg \left( \dfrac{1}{2}-e^{i\theta}\right)    -2\Arg(\delta -i\sin\theta)\right) ,\quad \delta  >0.
\end{aligned}
\end{equation}
Then,
\begin{equation}
\begin{aligned}
\int _{\Sigma _\theta (\delta  )}\log \left|\left(1+\frac{\log (y+e^{i\theta})}{\log r}\right)  \right|\frac {dy} {y}&=
\int _{\Sigma _\theta (\delta  )}\Big(\log \left|\left(1+\frac{\log (y+e^{i\theta})}{\log r}\right)  \right|\\
&-\log \left|\left(1+\frac{\log (e^{i\theta})}{\log r}\right)  \right|\Big)\frac {dy} {y}
+ i \nu , 
\end{aligned}
\end{equation}
with \begin{equation*}
      \nu = \log \left|\left(1+\frac{\log (e^{i\theta})}{\log r}\right)  \right| \left(2  \Arg \left( \dfrac{1}{2}-e^{i\theta}\right) - 2\Arg(\delta-i\sin\theta) \right).
     \end{equation*}

\comM{Since $cos(\theta)> 2/3$, for $y\in \mathcal Q$, $\log (y+e^{i\theta})$ is bounded away from zero so that}
we can define $r_0$ small enough in order to have:
\begin{eqnarray*}
\left|\frac{\log (y+e^{i\theta})}{\log r_0}\right|<1/2,\,\,\,\forall y\in \mathcal Q.
\end{eqnarray*}

\comM{Then,} for  all $y\in \mathcal Q$,
$$
\left|1+\frac{\log (y+e^{i\theta})}{\log r}\right|\ge 1-\left|\frac{\log (y+e^{i\theta})}{\log r}\right|>1-\left|\frac{\log (y+e^{i\theta})}{\log r_0}\right|>1/2.
$$
\comM{
We denote by 
\begin{equation*}
 v(\tau) = \log \left|\left(1+\frac{\log (\tau y+e^{i\theta})}{\log r}\right)  \right|,\quad v'(\tau )=\Re e\Bigg( \frac {y} {({\color{black}\tau}y+e^{i\theta})\log r \left(1+\frac{\log ({\color{black}\tau}y+e^{i\theta})}{\log r}\right) } \Bigg),
\end{equation*}
so that we use the mean value theorem to write  $v(1) -v(0) =v'(\tau) $ for some $\tau \in [0,1]$, i.e,
\begin{equation*}
\begin{aligned}
\log \left|\left(1+\frac{\log (y+e^{i\theta})}{\log r}\right)  \right|-\log \left|\left(1+\frac{\log (e^{i\theta})}{\log r}\right)  \right|
&=\Re e\Bigg( \frac {y} {({\color{black}\tau}y+e^{i\theta})\log r \left(1+\frac{\log ({\color{black}\tau}y+e^{i\theta})}{\log r}\right) } \Bigg)
\\&\leq 
\dfrac{|y|}{|\log(r)||\tau y +e^{i \theta}|\Big|1+ \dfrac{\log(\tau y+e^{i \theta})}{\log(r)} \Big|}.
\end{aligned}
\end{equation*}
}
For all $y\in \Gamma  _{ \theta }(\delta  )$,
\comM{
\begin{eqnarray*}
|\tau y +e^{i\theta}|\ge \tau \Re e (y)+\cos\theta\ge 
\tau \left(\dfrac{1}{2}-\cos\theta \right) +\cos\theta
 \ge- \dfrac{\tau}{2} + \dfrac{2}{3} \ge \dfrac{1}{6},
\end{eqnarray*}
\comM{and since} $\tau y  \in \mathcal Q$, we obtain {\color{black} for $r$ small enough}
\begin{eqnarray*}
\left|\log \left|\left(1+\frac{\log (y+e^{i\theta})}{\log r}\right)  \right|-\log \left|\left(1+\frac{\log (e^{i\theta})}{\log r}\right)  \right|\right|\le
\frac {12 |y|} {|\log r|},
\end{eqnarray*}
}
from where it  follows that, for all $\delta  >0$:
\begin{equation}
\begin{aligned}
\Big|\int _{\Sigma _\theta (\delta  )}&\log \left|\left(1+\frac{\log (y+e^{i\theta})}{\log r}\right)  \right|  \frac {dy} {y}\Big|\\
&\le 
\int _{\Sigma _\theta (\delta  )}\left|\log \left|\left(1+\frac{\log (y+e^{i\theta})}{\log r}\right)  \right|-\log \left|\left(1+\frac{\log (e^{i\theta}}{\log r}\right)  \right|\right|\frac {dy} {|y|} \comM{+ |\nu|} \nonumber\\
&\le \frac {\comM{12}} {|\log r|}\int _{\Sigma _\theta (\delta  )}dy \comM{+ 4 \pi\log  \left|\left(1+\frac{\log (e^{i\theta})}{\log r}\right)  \right| } 
\\
&=\frac{\comM{12}(2\cos\theta-1-2\delta ) }{|\log r|} \comM{+ 4 \pi\log  \left|\left(1+\frac{\log (e^{i\theta})}{\log r}\right)  \right| } .  \label{S6E124}
\end{aligned}
\end{equation}
On the other hand, for all $y\in \Delta _\theta (\delta  )\setminus \Sigma  _{ \theta }(\delta  )$, 
\begin{eqnarray*}
|y|\ge |\Re e (y)|\ge \cos \theta -1/2\ge 1/6.
\end{eqnarray*}
\begin{eqnarray}
\left|\int\limits _{\Delta _\theta (\delta  )\setminus \Sigma  _{ \theta }(\delta  )}\log \left|\left(1+\frac{\log (y+e^{i\theta})}{\log r}\right)  \right|\frac {dy} {y}\right|\le 6\left|\log \left(1-\frac{\log (2)}{\log r}\right)  \right|\int  _{ \Gamma _\theta }|dy|\nonumber\\
=\comM{9}\left|\log \left(1-\frac{\log (2)}{\log r}\right)  \right|. \label{S6E125}
\end{eqnarray}
Since by definition of $\Delta _\theta(\delta )$ and $\Sigma  _{ \theta }(\delta )$:
\begin{eqnarray*}
\int _{\Delta _\theta(\delta  )}\log \left|\left(1+\frac{\log (y+e^{i\theta})}{\log r}\right)  \right|\frac {dy} {y}
=\int _{\Sigma  _\theta(\delta  )}\log \left|\left(1+\frac{\log (y+e^{i\theta})}{\log r}\right)  \right|\frac {dy} {y}+\\
+\int _{\Delta _\theta(\delta  )\setminus \Sigma  _\theta(\delta  )}\log \left|\left(1+\frac{\log (y+e^{i\theta})}{\log r}\right)  \right|\frac {dy} {y}
\end{eqnarray*}
we deduce from \eqref{S6E124} and \eqref{S6E125} that for all $ \delta >0$
\begin{equation}
\begin{aligned}
\left|\int _{\Delta _\theta(\delta  )}\log \left|\left(1+\frac{\log (y+e^{i\theta})}{\log r}\right)  \right|\frac {dy} {y}\right|\le 
\frac{12(2\cos\theta-1)}{|\log r|}&+\comM{9}\left|\log \left(1-\frac{\log (2)}{\log r}\right) \right| \\
&\comM{+ 4 \pi\log  \left|\left(1+\frac{\log (e^{i\theta})}{\log r}\right)  \right| } ,
\end{aligned}
\end{equation}
and then, by \eqref{S6EVP} and \eqref{S6E123}:
\begin{equation}
\label{S6E140}
\int _{1/2}^{2} \log \left|\left(1+\frac{\log x}{\log r}\right)  \right|\frac {dx} {x-e^{i\theta}}=\mathcal O\left(\frac {1} {|\log r|}\right),\,\,\hbox{as}\,\,\,r\to 0.
\end{equation}

Suppose now that $\cos\theta<2/3$. In that case  
\begin{equation*}
 \comM{\left|x-e^{i\theta}\right|= \sqrt{(\Re e(x)) ^2-2\Re e(x) \cos\theta +1}\ge \sqrt{(\Re e(x)-2/3) ^2+5/9} \ge \sqrt 5 /3,}
\end{equation*}

and the denominator of the integral is bounded away from zero. 
Let $r_1$ be small enough in order to have:
$$
\left|\frac{\log x}{\log r_1}\right|<1/2,\,\,\forall x\in (1/2, 2).
$$
Then, for all $r<r_1$:
$$
 \left|1+\frac{\log x}{\log r}\right|\ge 
1-\left|\frac{\log x}{\log r}\right|\ge 1-\left|\frac{\log x}{\log r_1}\right|>1/2,\,\,\forall x\in (1/2, 2)
$$
\comM{By the mean value theorem applied to}
\begin{equation*}
 \color{black}{f(\tau)=\log\left(1+\tau \f{\log x}{\log r}\right),
 \quad f'(\tau)=\f{\log x}{\log r}\f{1}{1+\tau\f{\log x}{\log r}},
 }
\end{equation*}
\comM{we have for some $\tau \in [0,1]$}
\begin{equation*}
 \comM{
\left| \log\left|1+\frac{\log x}{\log r}\right| \right|=
\left|\f{\log x}{\log r}\f{1}{1+\tau\f{\log x}{\log r}}\right|
\leq \f{2\log(2)}{\vert \log(r)\vert} 
,
}\end{equation*}

and we deduce that for all $\theta$ such that $\cos\theta<2/3$:
\begin{equation}
\label{S6E141}
\left|\int _{1/2}^{2} \log \left|\left(1+\frac{\log x}{\log r}\right)  \right|\frac {dx} {x-e^{i\theta}}\right|
\le \frac {\comM{{\color{black}9  \log 2}}} {\comM{\sqrt{5}}|\log r|}
=\mathcal O\left(\frac {1} {|\log r|}\right),\,\,\hbox{as}\,\,\,r\to 0.
\end{equation}

As a consequence, by \eqref{S6E245}:
\begin{equation}\label{C4}
\begin{aligned}
\comM{I_{1,2}(\zeta)} &= \int _{r/2}^{2r} \log \left|\log w  \right|\frac {dw} {w-\zeta }\\
&=\log \left|\log r \right| \left(Log\left(2-\frac {\zeta } {|\zeta |}\right)-Log\left(\frac {1} {2}-\frac {\zeta } {|\zeta |}\right)\right)+\mathcal O\left(\frac{1}{\log|\zeta |}\right),\,\,|\zeta |\to 0.
\end{aligned}\end{equation}

{\subsubsection*{Summing the contributions $I(\zeta)=I_2(\zeta)+I_3(\zeta)+I_{1,1}(\zeta)+I_{1,2}(\zeta)+I_{1,3}(\zeta)$}}

Using   \eqref{S6E500} and adding  \eqref{S6E501}, \eqref{C1}, \eqref{C2} and \eqref{C4} we obtain:
 \begin{eqnarray}
 \label{I1aux1}
&& I _{ }(\zeta )= \frac {(\log 2)^2} {2} + \int_1^\infty\frac {\log \left|\log w  \right|} {w(w+1)}dw+\log \left|\log (r/2)  \right|\Log\left(1-\frac {\comM{|\zeta|}} {2\comM{\zeta}} \right)\nonumber\\
&&-\log(2r)(-1+\log|\log(2r)|) -\log|\log(2r)|\Log\left(1-\frac {\comM{\zeta}} {2\comM{|\zeta|}} \right)\nonumber\\
 &&+\log \left|\log r \right| \left(\Log\left(2-\frac {\zeta } {|\zeta |}\right)-\Log\left(\frac {1} {2}-\frac {\zeta } {|\zeta |}\right)\right)
 +\mathcal O\left(\frac {1} {\log|\zeta |}\right),\,\,\,|\zeta |\to 0.
 \end{eqnarray}
Using
\begin{equation}
\label{S6E289}
 \log \left|\log (r/2)  \right|=
 \log \left|\log r \right|+\mathcal O\left(\comM{1}\right), \quad \log \left|\log (2r)  \right|=
 \log \left|\log r \right|+\mathcal O\left(\comM{1}\right),\,\,\,|\zeta |\to 0
\end{equation}
the expression in \eqref{I1aux1} may be simplified to:
  \begin{eqnarray*}
&& I(\zeta )= \frac {(\log 2)^2} {2} + \int _1^\infty\frac {\log \left|\log w  \right|} {w(w+1)}dw+
\log \left|\log (r)  \right|\Log\left(1-\frac {|\zeta |} {2 \zeta } \right)\nonumber\\
&&-\log(r)(-1+\log|\log(r)|) -\log|\log(r)|\Log\left(1-\frac {\zeta } {2|\zeta |} \right)\nonumber\\
 &&+\log \left|\log r \right| \left(\Log\left(2-\frac {\zeta } {|\zeta |}\right)-\Log\left(\frac {1} {2}-\frac {\zeta } {|\zeta |}\right)\right)
 +\mathcal O\left(\comM{1}\right),\,\,\,|\zeta |\to 0.
 \end{eqnarray*}
 and
 \begin{eqnarray*}
&& I(\zeta )= \frac {(\log 2)^2} {2} + \int _1^\infty\frac {\log \left|\log w  \right|} {w(w+1)}dw-\log(r)(-1+\log|\log(r)|) +\nonumber\\
&&+\log|\log(r)| \left(\Log\left(1-\frac {|\zeta| } {2\zeta } \right)-\Log\left(1-\frac {\zeta } {2 |\zeta| } \right)+\right.\nonumber\\
 &&\left. +\Log\left(2-\frac {\zeta } {|\zeta |}\right)-\Log\left(\frac {1} {2}-\frac {\zeta } {|\zeta |}\right)\right)
 +\mathcal O\left(\comM{1}\right),\,\,\,|\zeta |\to 0.
 \end{eqnarray*}
Since:
$$
\Log\left(2-\frac {\zeta } {|\zeta |}\right)=\Log (2)+ \Log\left(1-\frac {\zeta } {2|\zeta |}\right)
$$
we first obtain:
\begin{eqnarray*}
\Log\left(1-\frac {|\zeta| } {2\zeta } \right)-\Log\left(1-\frac {\zeta } {2 |\zeta| } \right)+
\Log\left(2-\frac {\zeta } {|\zeta |}\right)-\Log\left(\frac {1} {2}-\frac {\zeta } {|\zeta |}\right)=\\
= \Log\left(1-\frac {|\zeta| } {2\zeta } \right)+\Log (2)-\Log\left(\frac {1} {2}-\frac {\zeta } {|\zeta |}\right).
\end{eqnarray*}
We use now:
\comM{
\begin{equation*}
\begin{aligned}
 {\color{black} p(\theta)}=\Log\left(1-\frac {|\zeta| } {2\zeta } \right)-\Log\left(\frac {1} {2}-\frac {\zeta } {|\zeta |}\right)
 &= \Log\left|\frac {1-\frac {|\zeta |} {2\zeta }} {\frac {1} {2}-\frac {\zeta } {|\zeta |}} \right|
+ i \left(\Arg\left(1-\frac {|\zeta| } {2\zeta } \right) - \Arg \left(\frac {1} {2}-\frac {\zeta } {|\zeta |}\right)\right)\\ 
{\color{black} p(\theta)}&=  i \left(\Arg\left(1-\frac {|\zeta| } {2\zeta } \right) - \Arg \left(\frac {1} {2}-\frac {\zeta } {|\zeta |}\right)\right) \\
& = i (\pi-\theta).
\end{aligned}
\end{equation*}}
{\color{black} Indeed, we have by definition $\zeta=\vert \zeta\vert e^{i\theta}$ with $\theta\in (0,2\pi).$ Let us denote $$\beta=1-\f{\vert \zeta\vert}{2\zeta}=1-\f{1}{2}e^{-i\theta}=1-\f{1}{2}\cos(\theta) +\f{i}{2}\sin(\theta)=\vert\beta\vert e^{i\alpha},\qquad \alpha\in (-\f{\pi}{6},\f{\pi}{6}).$$ We have
$$p(\theta)=i\left(\Arg(\beta) - \Arg \left(-\beta e^{i\theta} \right)\right)=i\left(\alpha - \Arg (e^{i(\theta -\pi+\alpha)})\right).$$
For $\theta\in (0,\pi)$ we have $\alpha\in (0,\f{\pi}{6})$, so that $\theta-\pi+\alpha \in (-\pi,\f{\pi}{6})$. For $\theta \in (\pi,2\pi)$ we have $\alpha\in (-\f{\pi}{6},0)$ so that $\theta-\pi+\alpha \in (-\f{\pi}{6},\pi).$ In both cases, we thus have $\theta-\pi+\alpha\in (-\pi,\pi),$ so that
$$p(\theta)=i\left(\alpha-\Arg (e^{i(\theta -\pi+\alpha)})\right)=i(\alpha-(\theta-\pi+\alpha))=\comM{i(}\pi-\theta).$$}

\comM{We deduce }
\begin{eqnarray*}
I(\zeta)=\frac {(\log 2)^2} {2} + \int _1^\infty\frac {\log \left|\log w  \right|} {w(w+1)}dw-\log(r)(-1+\log|\log(r)|)
+\log \left|\log (r)  \right|\log 2+\\
+i\log \left|\log (r)\right| \comM{(\pi- \theta)}+\mathcal O\left(\comM{1}\right),\,\,\,|\zeta |\to 0
\end{eqnarray*}

$$
\Im m (I(\zeta ))=\log \left|\log (r)\right| \comM{(\pi- \theta)} +\mathcal O\left(\comM{1}\right),\,\,\,|\zeta |\to 0
$$
and this proves \eqref{Im1z1} of Lemma \ref{Lemma_estimate_log} for $|\zeta |\to 0$.
\begin{remark}
\label{S6R1}
Notice that the argument leading to \eqref{C4} also proves the following:
\end{remark}
\vspace{-0.5cm}
 \begin{equation}\label{C4bis}
\int _{r/2}^{2r} \log \left|\log w  \right|\frac {dw} {w-\zeta }
=\log \left|\log r\right|\comM{\left(  {\Log}\left(2-e^{i\theta}\right)-  {\Log}\left(\frac{1} {2}-e^{i\theta}\right)\right)+\mathcal O\left(1\right),}\,\,|\zeta |\to \infty.
\end{equation}

  {\bf \boxed{Step\,\,II.} Limit as $|\zeta |\to + \infty$.}
  We split as well the integral $I(\zeta)$ in two terms
\begin{eqnarray}
I_1(\zeta )&=&J _{ 1 }(\zeta )+J _{ 2 }(\zeta )\label{Im1z2}\\
 J_{1}(\zeta )&=&\int _0^1 \log \left|\log w  \right|\left(\frac {1} {w-\zeta }-\frac {1} {w+1} \right)dw \label{Im1z3}\\
\quad J_{2}(\zeta )&=&\int _1^\infty \log \left|\log w  \right|\left(\frac {1} {w-\zeta }-\frac {1} {w+1} \right)dw \label{Im1z4}
\end{eqnarray}
and notice as well that $J_{1}(\zeta)$ converges toward a finite real limit:
\begin{equation}
\label{S6E1}
\lim _{|\zeta |\to \infty }J _{ 1 }(\zeta )=-\int _0^1 \frac {\log \left|\log w  \right|} {w+1}dw=\frac {\log (2)^2} {2}  
\end{equation}  
We write $J_{2}(\zeta)$ as the sum
\begin{align}
J_{2}(\zeta)&=\int _1^{r/2} \log \left|\log w  \right|\frac {\zeta +1} {(w-\zeta)(w+1) }dw
+\int _{r/2}^{2r} \log \left|\log w  \right|\frac {\zeta +1} {(w-\zeta)(w+1) }dw+\nonumber\\
&+\int _{2r}^\infty \log \left|\log w  \right|\frac {\zeta +1} {(w-\zeta)(w+1) }dw.\label{decomp}
\\
&\comM{=J_{2,1}+J_{2,2}+J_{2,3}.}
\end{align}
As previously, we examine the asymptotic behaviour of each of the three terms in the right hand side  of \eqref{decomp}. For the first term we notice the following:
\begin{align}\label{interm}
\comM{J_{2,1}=} \int _1^{r/2} \log \left|\log w  \right|\frac {\zeta +1} {(w-\zeta)(w+1) }dw&=\int _1^{r/2} \log \left|\log w  \right|\frac {\zeta +1} {(w-\zeta)w }dw-\nonumber\\
&-\int _1^{r/2} \log \left|\log w  \right|\frac {\zeta +1} {(w-\zeta)w(w+1) }dw.
\\
& \comM{= J_{2,2,1} + J_{2,2,2}} \nonumber 
\end{align}
The Lebesgue's convergence Theorem guarantees the convergence of \comM{$J_{2,2,2}$} towards a real finite limit:
\begin{equation}
\label{interm2}
\lim_{ |\zeta |\to \infty}\int _1^{r/2} \log \left|\log w  \right|\frac {\zeta +1} {(w-\zeta)w(w+1) }dw=
-\int _1^{\infty} \frac {\log \left|\log w  \right|} {w(w+1) }dw 
\end{equation}
On the other hand, the first term of \eqref{interm} can be written as
\begin{eqnarray*}
\comM{J_{2,2,1} =}\int _1^{r/2} \log \left|\log w  \right|\frac {\zeta +1} {(w-\zeta)w }dw
 &=&\int _1^{r/2} \log \left|\log w  \right|\frac {\zeta +1} {\zeta \left(\frac {w} {\zeta }-1\right)w }dw\\
 &=&- \sum_{ k=0 } ^\infty \frac {\zeta +1} {\zeta ^{k+1}}\int _1^{r/2} \log \left|\log w  \right|w^{k-1}dw\\
&&\hskip -3cm =-\frac {\zeta +1} {\zeta }\int _1^{r/2} \log \left|\log w  \right|w^{-1}dw- \sum_{ k=1 } ^\infty \frac {\zeta +1} {\zeta ^{k+1}}\int _1^{r/2} \log \left|\log w  \right|w^{k-1}dw.
\end{eqnarray*}
Using \comM{\eqref{app1_4} for $k \ge 1$}, we find
 \begin{eqnarray*}
 &&\int _1^{r/2} \log \left|\log w  \right|w^{k-1}dw=\frac {(r/2)^k\log(\log(r/2)))} {k}-\frac {\Psi (k\log(r/2))} {k}+\frac {\gamma_E+\log k} {k}\\
 &&=\left(\frac {r} {2}\right)^k\frac {\log(\log(r/2)))} {k}-
 \left(\frac {r} {2} \right)^k\frac {1} {k^2\log(r/2)}\left(1+\mathcal O \left(\frac {1} {k\log(r/2)} \right) \right)+\frac {\gamma_E+\log k} {k},
 \,\hbox{as}\,\,|\zeta |\to \infty,
\end{eqnarray*}
and then the first term in the right hand side of \eqref{interm} satifies:
\begin{eqnarray*}
&&\comM{J_{2,2,1}=}\int _1^{r/2} \log \left|\log w  \right|\frac {\zeta +1} {(w-\zeta)w }dw
 =-\frac {\zeta +1} {\zeta }\log(r/2)(-1+\log(\log(r/2)))-\nonumber\\
&&\qquad -\frac {\zeta +1} {\zeta }\log(\log(r/2))\sum_{ k=1 } ^\infty \left(\frac {r} {2\zeta }\right)^k\frac {1} {k}
 +\frac {\zeta +1} {\zeta \log(r/2)}\sum_{ k=1 } ^\infty \left(\frac {r} {2 \zeta } \right)^k\frac {1} {k^2}\left(1+\mathcal O \left(\frac {1} {k\log(r/2)} \right) \right)\comM{-}\nonumber\\
 &&\hskip 5.3cm \textcolor{black}{-\sum_{ k=1 } ^\infty \frac {(\zeta +1)} {\zeta ^{k+1}}\frac {(\gamma_E+\log k)} {k}},\,\,\hbox{as}\,\,|\zeta |\to \infty \nonumber\\
&&=-\frac {\zeta +1} {\zeta }\log(r/2)(-1+\log(\log(r/2)))
+\frac {\zeta +1} {\zeta }\log(\log(r/2))\Log\left(1-\frac {e^{-i\theta}} {2}\right)+\nonumber\\
&&\qquad+\frac {\zeta +1} {\zeta \log(r/2)}\sum_{ k=1 } ^\infty \left(\frac {r} {2 \zeta } \right)^k
\frac {1} {k^2}\left(1+\mathcal O \left(\frac {1} {k\log(r/2)} \right) \right) -\textcolor{black}{\frac {\zeta +1} {\zeta }\left(\gamma_E \Log\left(\frac {\zeta -1} {\zeta } \right)+\right.}\nonumber\\
&&\qquad\textcolor{black}{\left. + PolyLog^{(1, 0)}
\left[ 1, \frac {1} {\zeta }\right] \right)}
\end{eqnarray*}
Using again \eqref{S6E289}:
\begin{eqnarray}
\int _1^{r/2} \log \left|\log w  \right|\frac {\zeta +1} {(w-\zeta)w }dw&=&-\log(r)(-1+\log(\log(r)))+
\nonumber\\
&&\hskip -1cm +\log(\log(r))\Log\left(1-\frac {e^{-i\theta}} {2}\right)+\mathcal O\left(\comM{1} \right),\,\,\hbox{as}\,\,|\zeta |\to \infty.\label{CC1}
\end{eqnarray}

We consider now the third term in the right hand side of \eqref{decomp}. We must split again the integral in two terms as follows:
\begin{eqnarray}\label{lalala}
\comM{J_{2,3}= }\int _{2r}^\infty \log \left|\log w  \right|\frac {\zeta +1} {(w-\zeta)(w+1) }dw=\int _{2r}^\infty \log \left|\log w  \right|\frac {\zeta +1} {(w-\zeta)w }dw-
\nonumber\\
-\int _{2r}^\infty \log \left|\log w  \right|\frac {\zeta +1} {(w-\zeta)w(w+1) }dw.
\end{eqnarray}
The second term of \eqref{lalala} converges to zero as $r\to \infty$. We write the first term as follows:
\begin{eqnarray}
\int _{2r}^\infty \log \left|\log w  \right|\frac {\zeta +1} {(w-\zeta)w}dw&=&
\int  _{ 2r }^{\infty} \log \left|\log w  \right|\frac {\zeta +1} {w^2 \left(1-\frac {\zeta } {w }\right) }dw \nonumber\\
&=&
\sum _{ k=0 }^\infty (\zeta +1)\zeta ^k \int  _{ 2r }^{\infty} \log \left|\log w  \right|\frac {dw} {w^{k+2} }.
\end{eqnarray}
Using \comM{\eqref{app1_5}}, we get

\begin{eqnarray*}
\int  _{ 2r }^{\infty} \log \left|\log w  \right|\frac {dw} {w^{k+2} }&=&\frac {(2r)^{-(k+1)}\log\log(2r)} {k+1}+\frac {(2r)^{-k+1}} {(k+1)^2\log (2r)} \left( 1+\mathcal O \left(\frac {1} {(k+1)\log (2r)} \right)\right)
\end{eqnarray*}
and
we obtain for the first term in the right hand side of \eqref{lalala}:
\begin{eqnarray*}
\int _{2r}^\infty \log \left|\log w  \right|\frac {\zeta +1} {(w-\zeta)w}dw
 &=&\frac {\zeta +1} {\zeta }\log\log(2r)\sum _{ k=0 }^\infty\left(\frac {\zeta } {2r} \right)^{k+1} \frac {1} {k+1}+\nonumber \\
 &&\hskip -1.5cm +
 \frac {\zeta +1} {\zeta \log (2r)}\sum _{ k=0 }^\infty \left(\frac {\zeta } {2r} \right)^{k+1}
  \frac {1} {(k+1)^2} \left( 1+\mathcal O \left(\frac {1} {(k+1)\log (2r)} \right)\right)\nonumber\\
&=&-\frac {\zeta +1} {\zeta }\log\log(2r)\Log\left(1-\frac {e^{i\theta}} {2} \right)+\nonumber\\
&&\hskip -1.5cm +\frac {\zeta +1} {\zeta \log (2r)}\sum _{ k=0 }^\infty \left(\frac {\zeta } {2r} \right)^{k+1}\frac {1} {(k+1)^2} \left( 1+\mathcal O \left(\frac {1} {(k+1)\log (2r)} \right)\right).
\end{eqnarray*}
After using again \eqref{S6E289} we deduce:
\begin{eqnarray}\label{CC2}
\int _{2r}^\infty \log \left|\log w  \right|\frac {\zeta +1} {(w-\zeta)w}dw=-\log\log(r)\Log\left(1-\frac {e^{i\theta}} {2} \right)+
\mathcal O\left( \frac {1} {\log r}\right),\,\,\,\hbox{as}\,\,r\to \infty.
\end{eqnarray}
Using now \eqref{S6E1}, \eqref{interm2},  \eqref{CC1},  \eqref{CC2}  and \eqref{C4bis} in Remark \ref{S6R1},
\begin{eqnarray*}
I(\zeta )=\frac {\log (2)^2} {2} + \int _1^{\infty} \frac {\log \left|\log w  \right|} {w(w+1) }dw -\log(r)(-1+\log(\log(r)))+
\log(\log(r))\Log\left(1-\frac {e^{-i\theta}} {2}\right)-\\
-\log\log(r)\Log\left(1-\frac {e^{i\theta}} {2} \right)+\log \left|\log r\right|\comM{\left( \Log\left(2-e^{i\theta}\right) - \Log \left(\frac{1} {2}-e^{i\theta}\right)\right)}+
\mathcal O\left(\comM{1}\right),\,\,\,\hbox{as}\,\,r\to \infty.
\end{eqnarray*}  
The same arguments as in Step I yield first
\begin{eqnarray}
I(\zeta )&=&\frac {\log (2)^2} {2} + \int _1^{\infty} \frac {\log \left|\log w  \right|} {w(w+1) }dw
-\log(r)(-1+\log|\log(r)|)+\nonumber \\
&+&\log \left|\log (r)  \right| \comM{(\pi-\theta)}+\mathcal O\left(
\comM{1}\right),\,\,\,|\zeta |\to \infty,
\end{eqnarray}
and  then, property \eqref{Im1z1} of Lemma \ref{Lemma_estimate_log} for $|\zeta |\to \infty$.
\qed

\section*{Glossary}

\begin{tabular}{l l}
\hline
 Arg  & Principal value of the argument of a complex number: Arg$(z) \in (- \pi,\pi]$\\
$\arg$  & argument of a complex number: $\arg(z) \in [0,2\pi]$  \\
$\a$  & Multiplicative constant of the fragmentation rate \\
$\g$  &  Power of the fragmentation rate  \\
$B(x)$ 	& Fragmentation rate\\
$f(t,x)$ 	& Density of particles \\
$g(x)$ 	&        Stationary profile \\
$G(s)$ 	&        Mellin transform of $g$ \\
$k(y,x)$ 	&  Fragmentation kernel\\
$k_0(z)$ 	& Rescaled fragmentation rate \\
$K_0(s)$ 	&  Mellin transform of $k_0$\\
Log  & Logarithm of a complex number: Log$(z) = \log(|z|) + i $Arg$(z)$\\
$\log$  & Logarithm of a complex number:  $\log(z) = \log(|z|) + i \arg(z)$ \\
$L^1(\R^+, \mu)$ &$\biggl\{f:\R^+\to \R,\quad \int\limits_0^\infty \vert f(x)  \vert d\mu (x) <\infty\biggr\}$\\
$L^1(\R^+)$& $L^1(\R^+,dx)$ \\
$\mathcal{M}(\Omega)$ 	&Set of bounded  (or finite) measures over $\Omega$\\
 $\mathcal{M}_{loc}(\Omega)$&Set of measures over $\Omega$ which are finite over the compact sets\\
$M[\mu]$ 	& Mellin transform of a measure $\mu$ defined as 
$M[\mu](s):= \dst\int_0^{+ \infty} x^{s-1} d \mu(x)$ for $s\in \mathbb{C}$.\\
$t \in \mathbb{R^+}$ & Time \\
$x \in \mathbb{R^+}$	& Size of particles \\
\hline
\end{tabular}

\begin{acknowledgments}
 We would like to thank warmly Dr. W.F. Xue, School of Biosciences, Kent University, 
 for leading us to this problem, sharing with us new experimental data, 
 and for numerous interesting discussions.\\
 M.E. is supported by DGES Grant MTM2014-52347-C2-1-R and Basque Government Grant IT641-13. M. D. and M.T. were supported by the ERC Starting Grant SKIPPER$^{AD}$ (number 306321). 
\end{acknowledgments}

\bibliographystyle{plain}
\bibliography{bibli_14}
\end{document}